\documentclass[11pt,a4paper,reqno]{amsart}
\usepackage{amssymb}
\usepackage{amsmath,amsthm}
\usepackage{txfonts}
\usepackage{mathrsfs}
\usepackage{bbm}
\usepackage[mathscr]{euscript}
\def\bc{\begin{center}}
\def\ec{\end{center}}
\def\be{\begin{equation}}
\def\ee{\end{equation}}
\def\F{\mathbb F}
\def\N{\mathbb N}

\def\R{\mathbb R}

\def\A{\mathbb A}

\usepackage{colortbl}

\newtheorem{lem}{Lemma}[section]

\newtheorem{dfn}[lem]{Definition}
\newtheorem{pro}[lem]{Proposition}
\newtheorem{thm}[lem]{Theorem}

\newtheorem{cor}[lem]{Corollary}
\theoremstyle{remark}
\newtheorem{rem}{Remark}
\numberwithin{equation}{section}

\begin{document}
\title[%Localized Birkhoff average in beta dynamical system
]
{{Diophantine approximation of the orbit of 1 in the dynamical system of beta expansions}}

%\thanks{This work is  supported by NSFC 10631040 and 10901066.}
%\thanks{$^\dag$ Corresponding author.}

\author[Bing Li, Tomas Persson, Baowei Wang and Jun Wu]{Bing Li$^{1,2}$, Tomas Persson$^3$, Baowei Wang$^{4}$ and Jun Wu$^{4}$}
\address{$^1$ Department of Mathematics, South China University
of Technology, Guangzhou 510640,
  P. R. China}
  \address{$^2$ Departmentof Mathematical sciences,
University of Oulu,
 P.O. Box 3000 FI-90014, Finland}
 \email{libing0826@gmail.com}
 \address{$^3$ Center for Mathematical Sciences, Lund University, Box
118, 22100 Lund, Sweden.}
\email{tomasp@maths.lth.se}
\address{$^4$ School of Mathematics and Statistics, Huazhong University of Science and Technology, 430074 Wuhan, P. R. China.}
\email{bwei\_wang@yahoo.com.cn}
\email{jun.wu@mail.hust.edu.cn}



\begin{abstract}
  We consider the distribution of the orbits of the number $1$ under
  the $\beta$-transformations $T_\beta$ as $\beta$ varies.  Mainly,
  the size of the set of $\beta>1$ for which a given point can be well
  approximated by the orbit of $1$ is measured by its Hausdorff
  dimension. That is, the dimension of the following set
%The Hausdorff dimension of the well-approximable set
  \[
  E\big(\{\ell_n\}_{n\ge 1}, x_0\big)=\Big\{\,\beta>1:
  |T^n_{\beta}1-x_0|<\beta^{-\ell_n}, \ {\text{for infinitely
      many}}\ n\in \N\,\Big\}
  \]
  is determined, where $x_0$ is a given point in $[0,1]$ and
  $\{\ell_n\}_{n\ge 1}$ is a sequence of integers tending to infinity
  as $n\to \infty$.  For the proof of this result, the notion of the
  recurrence time of a word in symbolic space is introduced to
  characterise the lengths and the distribution of cylinders (the set
  of $\beta$ with a common prefix in the expansion of 1) in the
  parameter space $\{\,\beta\in \R: \beta>1\,\}$.
%{\color{blue} Q: Should we just write 'self-admissible word' to 'word' since 'self-%admissible' is not so clear for the reader in the abstract? }
%As a byproduct,
%we show that for a fixed $\beta>1$, any $n+1$ consecutive basic intervals of
%order $n$ contains at least one basic interval with length larger than $\beta^{-n}$.
\end{abstract}

\maketitle

\section{Introduction}

The study of Diophantine properties of the orbits in a dynamical
system has recently received much attention. This study contributes to
a better understanding of the distribution of the orbits in a
dynamical system. Let $(X, \mathcal{B}, \mu, T)$ be a
measure-preserving dynamical system %(probability space)
with a
consistent metric $d$.
%Assume that $\mu$ is ergodic, then %for any measurable set $B\subset
%X$ with positive measure, the Birkhoff ergodic theorem %ensures that
%$\mu$-almost all points in $X$ will hit the target $B$ infinitely
%often under the iteration of $T$. That
%is, \begin{equation}\label{Birkhoffresults} \mu\{x\in X: T^nx\in
%B\ \mbox{infinitely often (i.o.)}\}=1.  \end{equation} Let $d$ be a
%metric on $X$ consistent with the probability space
%$(X,\mathcal{B},\mu)$.  The result \eqref{Birkhoffresults}
If $T$ is ergodic with respect to the measure $\mu$, then a
consequence of Birkhoff's ergodic theorem is the following
hitting property, namely, for any $x_0\in
X$ and $\mu$-almost all $x\in X$,
\begin{equation}\label{1.2}
  \liminf_{n\to\infty} d(T^n(x),x_0)=0.
\end{equation}
%But what will happen if the target $B$ shrinks with the time $n$ in
%\eqref{Birkhoffresults} and
One can then ask, what are the quantitative properties of
the convergence speed in \eqref{1.2}? More precisely, for a given
sequence of balls $B(x_0,r_n)$ with center $x_0\in X$ and shrinking
radius $\{r_n\}$, what are the metric properties of the set
\[
F(x_0,\{r_n\}):=\Big\{x\in X: d(T^nx, x_0)< r_n\ {\text{for
    infinitely many}}\ n\in \mathbb{N}\Big\}
\]
in the sense of measure and in the sense of dimension?  Generally, let
$\{B_n\}_{n\ge 1}$ be a sequence of measurable sets with $\mu(B_n)$
decreasing to $0$ as $n\to\infty$. The problem concerning the metric
properties of the set
\begin{equation}\label{11}
  \Big\{\, x\in X: T^nx\in B_n\ {\text{for infinitely many}}\ n\in
  \mathbb{N} \,\Big\}
\end{equation}
is named as the dynamical Borel--Cantelli Lemma (see \cite{ChK}) or
shrinking target problem \cite{HV}.
%If $B_n$ shrinks with the time, the corresponding problem is called shrinking
%target problem.
%When there is a metric $d$ on $X$
%consistent with the probability space $(X, \mathcal{B}, \mu)$ and $\{B(y_0, r_n)\}$ is a sequence of balls with shrinking radius,

% Given a sequence of balls $B(x_0,r_n)$ with center $x_0\in X$ and shrinking radius $\{r_n\}$, the set
%$$F(x_0,\{r_n\}):=\{x\in X: d(T^nx, x_0)< r_n\ \mbox{i.o.}\}$$
%is analogous with the classic Diophantine approximation in some sense. So the set $F(x_0, \{r_n\})$
%is also called well-approximable set by Hill and Velani \cite{HV}.

In this paper, we consider a modified shrinking target problem. Let us
begin with an example to illustrate the motivation.
%{\color{blue} (Necessary of this paragraph???) When $\sum_{n=1}^\infty\mu(B(y_0,r_n))<+\infty$,
%the convergence part of the classic Borel-Cantelli Lemma indicates
%that $F(y_0,\{r_n\})$ is a $\mu$-null set. So, subsequently, it is natural to ask how large is the well-approximable set $F(y_0, \{r_n\})$
%from the sense of the Hausdorff dimension? Up to now, some elegant results are presented in some concrete systems.
%Hill and Velani studied the dimension of the set of the shrinking target problems
%\[
%\Big\{x\in X: |T^nx-z_0|< e^{-\sum_{j=0}^{n-1}f(T^jx)} \ {\text{for infinitely many}}\ n\in \N\Big\},
%\]
%in the system $(X, T)$ with $T$ an expanding rational map with degree $\geq 2$ and $X$ the corresponding Julia set \cite{HV,HV97},
%as well as in the system $(X,T)$ with $X$ an $n$-dimensional torus and $T$ a linear map \cite{HV99}.
%Similar settings were also carried out for the actions on finite Kleinian groups by Hill and Velani \cite{HilV}
%and parabolic rational maps by Stratmann and Urba\'{n}ski \cite{StU}.
%Later, the results in the conformal iterated function systems were established by Urba\'{n}ski \cite{Urb}.
%Recently, Fern\'{a}ndez, Meli\'{a}n and Pestana \cite{FMP1} set up general theories for Markov expanding systems.
%For the irrational rotations on torus, see the work of Bugeaud \cite{Bu}, Troubetzkoy and Schmeling \cite{ST},
%Fan and Wu \cite{FaW}, Bugeaud, Harrap, Kristensen and Velani \cite{BHKV}.}
Let $R_\alpha \colon x \mapsto x+\alpha$ be a rotation map on the unit
circle.  Then the set studied in classical inhomogeneous Diophantine
approximation can be rewritten as
\begin{equation}\label{fff4}
  \Big\{\, \alpha\in \mathbb{Q}^c: |R_{\alpha}^n0-x_0|<r_n, \ {\text{for
      infinitely many}}\ n\in \mathbb{N} \,\Big\}.
\end{equation}
The size of the set in (\ref{fff4}) in the sense of Hausdorff measure
and Hausdorff dimension was studied by Bugeaud \cite{Bu1},
Levesley \cite{Lev}, Bugeaud and Chevallier \cite{Bu2} etc. Compared
with the shrinking target problem (\ref{11}), instead of considering
the Diophantine properties in {\em one} given system, the set (\ref{fff4})
concerns the properties of the orbit of some given point (the orbit of
0) in {\em a family} of dynamical systems. It is the set of parameters
$\alpha$ such that $R_\alpha$ share some common property.

Following this idea, in this paper, we consider the same setting as
(\ref{fff4}) in the dynamical systems $([0,1], T_{\beta})$ of
$\beta$-expansions with $\beta$ varying in the {\em parameter space}
$\{\, \beta\in \R : \beta>1 \,\}$.
% named by Persson and Schmeling \cite{PerS}.

%Bugeaud \cite{Bu}, Schmeling and Troubetzkoy \cite{ST}
%considered that for the irtional rotation dynamical system (Diophantine approximation);
%Hill and Velani \cite{HV99} studied that for the linear map given by a matrix with integer
%coefficients on an $n$-dimensional torus; Shen and Wang \cite{SW} considered that for beta-dynamical systems etc.

It is well-known that $\beta$-expansions are typical examples of
one-dimensional expanding systems, whose information is reflected by
some critical point. In the case of $\beta$-expansion, this critical
point is the unit $1$. This is because the $\beta$-expansion of 1 (or
the orbit of 1 under $T_\beta$) can completely characterize all
admissible sequences in the $\beta$-shift space (see \cite{Pa}), the
lengths and the distribution of cylinders induced by $T_\beta$
\cite{FW}, etc.
 %For example, the criterion of admissible sequence \cite{Pa}, the length and the %distribution of cylinders \cite{FW},
 %etc are all tightly related to the expansion of 1 %(or the orbit of 1 under $T_\beta$).
Upon this, in this current work, we study the Diophantine properties
of $\{T^n_{\beta}1\}_{n\ge 1}$, the orbit of 1, as $\beta$ varies in
the parameter space $\{\, \beta\in \R: \beta>1\,\}$.

Blanchard \cite{Bl} gave a kind of classification of the parameters in the
space $\{\,\beta\in \R: \beta>1\,\}$ according to the distribution of
$\mathscr{O}_{\beta}:=\{T^n_{\beta}1\}_{n\ge 1}$: (i) ultimately zero;
(ii) ultimately non-zero periodic; (iii) $0$ is not an accumulation
point of $\mathscr{O}_\beta$ (exclude those $\beta$ in classes (i,ii)
); (iv) non-dense in $[0,1]$ (exclude $\beta$'s in classes
(i,ii,iii)); and (v) dense in $[0,1]$. It was shown by Schmeling
\cite{Schme} that the class (v) is of full Lebesgue measure (the
results in \cite{Schme} give more, that for almost all $\beta$, all
allowed words appear in the expansion of $1$ with regular
frequencies).  This dense property of $\mathscr{O}_{\beta}$ for almost
all $\beta$ gives us a type of hitting property, i.e., for any $x_0\in
[0,1]$,
\begin{equation}\label{1.1} \liminf_{n\to\infty}
  |T^n_{\beta}1-x_0|=0, \quad \text{for } \mathcal{L}{\text{-a.e.}}  \ \beta>1,
\end{equation}
where $|x-y|$ means the distance between $x,y\in \mathbb{R}$, and
$\mathscr{L}$ is the Lebesgue measure on $\mathbb{R}$. Similarly as
for (\ref{1.2}), we would like to know the speed of convergence in
(\ref{1.1}).

% In this paper, we study the Diophantine properties of $\{T^n_{\beta}1\}_{n\ge 1}$, the orbit of 1, as $\beta$ varies.
Fix a point $x_0\in [0,1]$ and a sequence of positive integers
$\{\ell_n\}_{n\ge 1}$. Consider the set of $\beta>1$ for which $x_0$
can be well approximated by the orbit of 1 under the
$\beta$-expansions with given shrinking speed, namely the
set
\begin{equation}\label{1.3}
  E\big(\{\ell_n\}_{n\ge 1}, x_0\big)=\Big\{\, \beta>1:
  |T^n_{\beta}1-x_0|<\beta^{-\ell_n}, \ {\text{for infinitely
      many}}\ n\in \mathbb{N} \,\Big\}.
\end{equation}
This can be viewed as a kind of shrinking target problem in the
parameter space.

When $x_0=0$ and $\ell_n=\gamma n (\gamma > 0)$, Persson and Schmeling
\cite{PerS} proved that
\[
\dim_\textsf{H}E(\{\gamma n\}_{n\ge 1}, 0)=\frac{1}{1+\gamma},
\]
where $\dim_\textsf{H}$ denotes the Hausdorff dimension. For a general
$x_0\in [0, 1]$ and a sequence $\{\ell_n\}$, we have the following.

\begin{thm}\label{maintheorem}
  Let $x_0\in [0,1]$ and let $\{\ell_n\}_{n\ge 1}$ be a sequence of
  positive integers such that $\ell_n\to\infty$ as $n\to\infty$. Then
  \[
  \dim_\textsf{H}E\big(\{\ell_n\}_{n\ge 1},
  x_0\big)=\frac{1}{1+\alpha}, \quad
  {\text{where}}\ \alpha=\liminf_{n\to\infty}\frac{\ell_n}{n}.
  \]
\end{thm}

In other words, the set in (\ref{1.3}) concerns points in the
parameter space $\{\,\beta>1: \beta\in \R\,\}$ for which the orbit
$\{\,T^n_{\beta}1: n\ge 1\,\}$ is close to  the same
magnitude $x(\beta)=x_0$ for infinitely many moments in time. What can
be said if the magnitude $x(\beta)$ is also allowed to vary
continuously with $\beta>1$?  Let $x=x(\beta)$ be a function on $(1,
+\infty)$, taking values on $[0, 1]$. The setting \eqref{1.3} changes
to
\begin{equation} \label{eq:Etilde}
  \widetilde{E}\big(\{\ell_n\}_{n\ge 1}, x\big)=\Big\{\,\beta>1:
  |T^n_{\beta}1-x(\beta)|<\beta^{-\ell_n}, \ {\text{for infinitely
      many}}\ n\in \mathbb{N}\,\Big\}.
\end{equation}
As will become apparent, the proof of Theorem~\ref{maintheorem} also
works for this general case $x=x(\beta)$ after some minor adjustments,
and we can therefore state the following theorem.

\begin{thm}\label{t2}
  Let $x=x(\beta): (1, +\infty)\to [0, 1]$ be a Lipschtiz continuous
  function and $\{\ell_n\}_{n\ge 1}$ be a sequence of positive
  integers such that $\ell_n\to\infty$ as $n\to\infty$. Then
  \[
  \dim_\textsf{H}\widetilde{E}\big(\{\ell_n\}_{n\ge 1},
  x\big)=\frac{1}{1+\alpha}, \quad
  {\text{where}}\ \alpha=\liminf_{n\to\infty}\frac{\ell_n}{n}.
  \]
\end{thm}

Theorems \ref{maintheorem} (as well as Theorem \ref{t2}) can be viewed
as a generalization of the result of Persson and Schmeling
\cite{PerS}. But there are essential differences between the three
cases when the target $x_0=0$, $x_0\in (0,1)$ and $x_0=1$. The
following three remarks serve as an outline of the differences.

\begin{rem}
  The generality of $\{\ell_n\}_{n\ge 1}$ gives no extra difficulty
  compared with the special sequence $\{\ell_n = \alpha n\}_{n\ge
    1}$. However, there are some essential difficulties when
  generalizing $x_0$ from zero to non-zero.  The idea used in
  \cite{PerS}, to construct a suitable Cantor subset of
  $E\big(\{\ell_n\}_{n\ge 1}, x_0\big)$ to get the lower bound of
  $\dim_\textsf{H}E(\{\ell_n\}_{n\ge 1}, x_0)$, is not applicable for
  $x_0\ne 0$. For any $\beta>1$, let
  \[
  \varepsilon_1(x,\beta),\varepsilon_2(x,\beta),\ldots
  \]
  be the digit sequence of the $\beta$-expansion of $x$. To guarantee
  that the two points $T^n_{\beta}1$ and $x_0$ are close enough, a
  natural idea is to require that
  \begin{equation}\label{fff5}
    \varepsilon_{n+1}(1,\beta)=\varepsilon_1(x_0,
    \beta),\ \ldots,\ \varepsilon_{n+\ell}(1,\beta)=\varepsilon_{\ell}(x_0,
    \beta)
  \end{equation}
  for some $\ell\in \N$ sufficiently large. When $x_0=0$, the
  $\beta$-expansions of $x_0$ are the same (all digits are 0) no
  matter what $\beta$ is. Thus to fulfill (\ref{fff5}), one needs only
  to consider those $\beta$
  %whose expansion of $1$ is followed by a long period of zeros after the $n$-th position.
  for which a long string of zeros follows $\varepsilon_n(1,\beta)$ in
  the $\beta$-expansion of 1.  But when $x_0\ne 0$, the
  $\beta$-expansions of $x_0$ under different $\beta$ are
  different. Furthermore, the expansion of $x_0$ is not known to us,
  since $\beta$ has not been determined yet.  This difference constitutes a main difficulty
  in constructing points $\beta$ fulfilling the conditions in
  defining $E\big(\{\ell_n\}_{n\ge 1}, x_0\big)$.

  To overcome this difficulty, a better understanding of the parameter
  space seems necessary. In Section~\ref{sec:distribution}, we analyse
  the length and the distribution of a cylinder in the parameter space
  which relies heavily on a  newly cited notion called {\em the recurrence time of
    a word}.
\end{rem}

\begin{rem}
  When $x_0\neq 1$, the set $E(\{\ell_n\}_{n\ge 1}, x_0)$ can be
  regarded as a type of shrinking target problem with fixed
  target. While when $x_0=1$, it becomes a type of recurrence
  properties. There are some differences between these two
  cases. Therefore, their proofs for the lower bounds of
  $\dim_\textsf{H}E(\{\ell_n\}_{n\ge 1}, x_0)$ are given separately in
  Sections 5 and 6.
\end{rem}

\begin{rem}
  If $x(\beta)$, when developed in base $\beta$, is the same for all
  $\beta\in (\beta_0, \beta_1)$, then with an argument based on
  Theorem 15 in \cite{PerS}, one can give the dimension of
  $\widetilde{E}(\{\ell\}_{n\ge 1}, x(\beta))$.  However as far as a
  general function $x(\beta)$ is concerned, the idea used in proving
  Theorem \ref{maintheorem} can also be applied to give a full
  solution of the dimension of $\widetilde{E}(\{\ell\}_{n\ge 1},
  x(\beta))$.
\end{rem}

%For more dimensional results related to beta-transformation dynamical system, we refer to the papers of  Pfister and Sullivan \cite{PfS},  Schmeling \cite{Schme},
%Thompson \cite{Tho},  F\"{a}rm,  Persson and  Schmeling \cite{FPS}, Tan and Wang \cite{TaW} and
%the references therein.
For more dimensional results related to the $\beta$-expansion, the readers
are referred to \cite{FPS,PfS,Schme, TaW, Tho} and references therein. %the papers of  Pfister and Sullivan \cite{PfS},  Schmeling \cite{Schme},
%Thompson \cite{Tho},  F\"{a}rm,  Persson and  Schmeling \cite{FPS}, Tan and Wang \cite{TaW} and
%the references therein.
For more dimensional results concerning the shrinking target problems,
see \cite{Bu, BHKV,FaW, FMP1,HV,HV97, HilV, HV99, ST, SW, StU, Urb}
and references therein.

\section{Preliminary}
This section is devoted to recalling some basic properties of
$\beta$-expansions and fixing some notation. For more
information on $\beta$-expansions see \cite{Bl,Hof,Pa,Re} and
references therein.

 \smallskip

The $\beta$-expansion of real numbers was first introduced by
R\'{e}nyi \cite{Re}, which is given by the following algorithm. For
any $\beta>1$, let
%Write
%$\mathcal{A}=\{0,1,\cdots,\beta-1\}$ when $\beta$ is an integer and otherwise,
%$\mathcal{A}=\{0,1,\cdots,\lfloor\beta\rfloor\}$,
%where $\lfloor x\rfloor$ is the integer part of $x\in\mathbb{R}$.
\begin{equation}\label{f2}
  T_{\beta}(0):=0, \quad  T_{\beta}(x)=\beta x-\lfloor\, \beta x\rfloor,
  \ x\in (0,1),
\end{equation}
where $\lfloor \xi\rfloor$ is the integer part of $\xi\in\mathbb{R}$.
 %where $\lceil \xi \rceil$ denotes the smallest integer no less than $\xi$.
By taking
\[
\varepsilon_n(x,\beta)=\lfloor\, \beta T_{\beta}^{n-1}x\rfloor\in \mathbb{N}
\]
recursively for each $n\ge 1$, every $x\in [0,1)$ can be uniquely
  expanded into a finite or an infinite
  sequence
\begin{equation}\label{a2}
  x=\frac{\varepsilon_1(x,\beta)}{\beta} + \cdots +
  \frac{\varepsilon_n(x,\beta)}{\beta^n} + \cdots,
\end{equation}
which is called the $\beta$-expansion of $x$ and the sequence
$\{\varepsilon_n(x,\beta)\}_{n\ge 1}$ is called the digit sequence of
$x$.
%We also write $\{0\}_{n\ge 1}$ as the digit sequence of $0$.
We also write (\ref{a2}) as $\varepsilon(x,\beta) =
(\varepsilon_1(x,\beta), \ldots, \varepsilon_n(x,\beta),\ldots)$. The
system $([0,1), T_{\beta})$ is called a {\it $\beta$-dynamical system}
  or a {\em $\beta$-system}.

\begin{dfn}
  A finite or an infinite sequence $(w_1, w_2, \ldots)$ is said to be
  {\it admissible} (with respect to the base $\beta$), if there exists
  an $x\in [0,1)$ such that the digit sequence (in the
    $\beta$-expansion) of $x$ equals $(w_1, w_2, \ldots)$.
\end{dfn}
Denote by $\Sigma_{\beta}^n$ the collection of all $\beta$-admissible
sequences of length $n$ and by $\Sigma_{\beta}$ that of all
infinite admissible sequences. Write
$\mathcal{A}=\{0,1,\ldots,\beta-1\}$ when $\beta$ is an integer and
otherwise, $\mathcal{A}=\{0,1,\ldots,\lfloor\beta\rfloor\}$. Let
$S_\beta$ be the closure of $\Sigma_\beta$ under the product topology
on $\mathcal{A}^{\mathbb{N}}$.  Then $(S_\beta,\sigma|_{S_\beta})$ is
a subshift of the symbolic space $(\mathcal{A}^{\mathbb{N}},\sigma)$,
where $\sigma$ is the shift map on $\mathcal{A}^{\mathbb{N}}$.
%$$\Sigma_{\beta} = \Big\{w\in \mathcal{A}^{\mathbb{N}}: {\rm{there\ exists}}\  x \
%{\rm{such\ that\ its}}\ \beta{\rm{-expansion\ is}}\ w\Big\}.$$
%
%
%
%For any $\beta$-admissible sequence $(w_1,\cdots,w_n)\in \Sigma^n_{\beta}$,   the set
%$$I_{n}(w_1,\ldots, w_n):=\Big\{x\in [0,1): w_j(x)=w_j, 1\leq j\leq n\Big\}
%$$ is called a basic   interval of order $n$ (with respect to the base $\beta$).
%
%\smallskip

Let us now turn to the {\em infinite $\beta$-expansion of 1}, which
plays an important role in the study of $\beta$-expansion.  At first,
apply the algorithm (\ref{f2}) to the number $x=1$. Then the number $1$ can also
be expanded into a series, denoted by
\[
1 = \frac{\varepsilon_1(1,\beta)}{\beta} + \cdots +
\frac{\varepsilon_n(1,\beta)}{\beta^n} + \cdots.
\]
If the above series is finite, i.e.\ there exists $m\ge1$ such that
$\varepsilon_m(1,\beta)\neq 0$ but $\varepsilon_n(1,\beta)=0$ for
$n>m$, then $\beta$ is called a simple Parry number. In this case, the
digit sequence of 1 is given as
\[
\varepsilon^*(1,\beta) := (\varepsilon_1^*(\beta),
\varepsilon_2^*(\beta), \ldots) = (\varepsilon_1(1,\beta), \ldots,
\varepsilon_{m-1}(1,\beta), \varepsilon_m(1,\beta)-1)^\infty,
\]
where $(w)^\infty$ denotes the periodic sequence $(w,w,w,\ldots)$.  If
$\beta$ is not a simple Parry number, the digit sequence of 1 is given
as
%we still use the notation $(\varepsilon_1^*(\beta),\varepsilon_2^*(\beta),\ldots)$ to express
 % the $\beta$-expansion of~$1$, that is,
\[
\varepsilon^*(1,\beta) := (\varepsilon_1^*(\beta),
\varepsilon_2^*(\beta), \ldots) = (\varepsilon_1(1,
\beta),\varepsilon_2(1, \beta),\ldots).
\]
In both cases, the sequence
$(\varepsilon_1^*(\beta),\varepsilon_2^*(\beta),\ldots)$ is called
{\em the infinite $\beta$-expansion of 1} and we always have
that
\begin{equation}
  1 = \frac{\varepsilon_1^*(\beta)}{\beta} + \cdots +
  \frac{\varepsilon_n^*(\beta)}{\beta^n}+\cdots.
\end{equation}

The lexicographical order $\prec$ between the infinite sequences is
defined as follows: $$w=(w_1, w_2, \ldots, w_n,\ldots)\prec w'=(w_1',
w_2', \ldots, w_n',\ldots)$$ if there exists $k \geq 1$ such that $w_j
=w_j'$ for $1\leq j<k$, while $w_k<w_k'$. The notation $w\preceq w'$
means that $w\prec w'$ or $w=w'$. This ordering can be extended to
finite blocks by identifying a finite block
$(w_1,\ldots,w_n)$ with the sequence $(w_1,\ldots,w_n,0,0,\ldots)$.

%Write $$
%1=\frac{w^*_{1}(1,\beta)}{\beta}+\cdots+\frac{w^*_{n}(1,\beta)}{\beta^n}+\cdots.
%$$
%$\bullet$ {\it Criterion for admissible sequence.}

The following result due to Parry \cite{Pa} is a criterion for the
admissibility of a sequence which relies heavily on the {\em infinite
  $\beta$-expansion of 1}.
\begin{thm}[Parry \cite{Pa}]\label{t4}
\

{\rm{(1)}} Let $\beta> 1$. For each $n\ge 1$, a block of non-negative
integers $w=(w_1,\ldots, w_n)$ belongs to $ \Sigma_{\beta}^n$ if and
only if
\[
\sigma^i w \preceq \varepsilon_1^*(1,\beta),\ldots,
\varepsilon_{n-i}^*(1,\beta)\ \ \mbox{for all}\ \ 0\le i<n.
\]
%$$(w_k,w_{k+1},\cdots) \prec (\varepsilon^*_1(\beta), \varepsilon^*_2(\beta), \cdots ) \ \ \mbox{for all} \ k\ge 1.$$
%As a consequent, $w\in S_\beta$ if and only if $\sigma^k w \preceq \varepsilon^*(1,\beta)$ for all $k\ge 0$.

{\rm{(2)}} The function $\beta \mapsto \varepsilon^*(1,\beta)$ is
increasing with respect to the variable $\beta>1$. Therefore, if
$1<\beta_1<\beta_2$, then
\[
\Sigma_{\beta_1}\subset \Sigma_{\beta_2},
\ \ \Sigma_{\beta_1}^n\subset \Sigma_{\beta_2}^n \ \ (\mbox{for all}
\ n\ge 1).
\]
\end{thm}

%As an application of the criterion of admissibility, we have the following assertion. %announcement.
% For each integer $n\in \N$, let
% \begin{equation}
%   t_n=\max\big\{k: w^*_{i+1}(\beta)=\cdots=w^*_{i+k}(\beta)=0, \ {\text{for some }}\ 1\le i\le n\big\}, \label{tn}
% \end{equation}
% that is the number of consecutive zeros beginning at one of the first $n$ places in the $\beta$-expansion of unity.
% \begin{pro}\label{p1} For any $n\in \N$ and $(w_1,\cdots,w_n)\in \Sigma_{\beta}^n$,
% $$
% \big|I_{n+t_n}(w_1,\cdots, w_n, 0^{t_n})\big|=\beta^{-n-t_n}
% $$
% \end{pro}
% \proof By the criterion of admissibility, for any $\sigma\in \Sigma_{\beta}$, the concatenation $
% (w_1,\cdots, w_n, 0^{t_n} , \sigma)
% $ is admissible. \hfill $\Box$

At the same time, Parry also presented a characterization of when a
sequence of integers is the expansion of 1 for some $\beta>1$. First,
we cite a notation: {\em self-admissible}.
\begin{dfn}
  A word $w = (\varepsilon_1, \ldots, \varepsilon_n)$ is called
  self-admissible if for all $1\le i <n$
  \[
  \sigma^i(\varepsilon_1,\ldots,\varepsilon_n)\preceq
  \varepsilon_1,\ldots,\varepsilon_{n-i}.
  \]
  An infinite digit sequence $w=(\varepsilon_1,\varepsilon_2, \ldots)$ is said to
  be self-admissible if for all $i\ge 1$, $\sigma^iw\prec w$.
\end{dfn}

\begin{thm}[Parry \cite{Pa}]
  A digit sequence $(\varepsilon_1, \varepsilon_2, \ldots )$ is the
 expansion of $1$ for some $\beta>1$ if and only if it is
  self-admissible.
\end{thm}
The following result of R\'{e}nyi implies that the dynamical system
$([0,1), T_{\beta})$ admits $\log \beta$ as its topological
  entropy.
  \begin{thm}[R\'{e}nyi \cite{Re}] \label{Re} Let
    $\beta>1$. For any $n\ge 1$,
    \[
    \beta^n\le \sharp \Sigma_{\beta}^n\le {\beta^{n+1}}/({\beta-1}),
    \]
    here and hereafter $\sharp$ denotes the cardinality of a finite
    set.
  \end{thm}

\section{Distribution of regular cylinders in parameter space} \label{sec:distribution}
From this section on, we turn to the parameter space $\{\,\beta\in \R:
\beta>1\,\}$, instead of considering a fixed $\beta>1$. We will address
the length of a cylinder in the parameter space, which is closely
related to the notion of {\em recurrence time}.

\begin{dfn}\label{cylinderpar}
  Let $(\varepsilon_1, \ldots, \varepsilon_n)$ be self-admissible.  A
  cylinder in the parameter space is defined as
  \[
  I_n^P(\varepsilon_1,\ldots,\varepsilon_n)=\Big\{\, \beta>1:
  \varepsilon_1(1,\beta)=
  \varepsilon_1,\ldots,\varepsilon_n(1,\beta)=\varepsilon_n \,\Big\},
  \]
  i.e., the set of $\beta$ for which the $\beta$-expansion of 1 begins
  with the common prefix $\varepsilon_1, \ldots,
  \varepsilon_n$. Denote by $C_n^P$ the collection of cylinders of
  order $n$ in the parameter space.
\end{dfn}

\subsection{Recurrence time of words}\label{sec3.1}

\begin{dfn}
  Let $w=(\varepsilon_1, \ldots, \varepsilon_n)$ be a word of length
  $n$. The recurrence time $\tau(w)$ of $w$ is defined as
  \[
  \tau(w):=\inf\big\{\, k\ge 1:
  \sigma^k(\varepsilon_1,\ldots,\varepsilon_n) =
  (\varepsilon_1,\ldots,\varepsilon_{n-k}) \,\big\}.
  \]
  If such an integer $k$ does not exist, then $\tau(w)$ is defined to
  be $n$ and $w$ is said to be of full recurrence time.
\end{dfn}

Applying the definition of recurrence time and the criterion of
self-admissibility of a sequence, we obtain the following.
\begin{lem}\label{lem3.3}
  Let $w=(\varepsilon_1,\ldots, \varepsilon_n)$ be self-admissible
  with the recurrence time $\tau(w)=k$. Then for each $1\le
  i<k$,
  \begin{equation}\label{ff1}
    \varepsilon_{i+1},\ldots, \varepsilon_k \prec
    \varepsilon_1,\ldots,\varepsilon_{k-i}.
  \end{equation}
\end{lem}

\begin{proof}
  The self-admissibility of $w$ ensures that
  \[
  \varepsilon_{i+1},\ldots,\varepsilon_k, \varepsilon_{k+1},\ldots,
  \varepsilon_n\preceq \varepsilon_1,\ldots, \varepsilon_{k-i},
  \varepsilon_{k-i+1},\ldots, \varepsilon_{n-i}.
  \]

  The recurrence time $\tau(w)=k$ of $w$   implies that for $1\le i<k$,
  \[
  \varepsilon_{i+1},\ldots,\varepsilon_k, \varepsilon_{k+1},\ldots,
  \varepsilon_n\neq \varepsilon_1,\ldots, \varepsilon_{k-i},
  \varepsilon_{k-i+1},\ldots, \varepsilon_{n-i}.
  \]
  Combining the above two facts, we arrive at
  \begin{equation}\label{f9}
    \varepsilon_{i+1},\ldots,\varepsilon_k, \varepsilon_{k+1},\ldots,
    \varepsilon_n \prec \varepsilon_1,\ldots, \varepsilon_{k-i},
    \varepsilon_{k-i+1},\ldots, \varepsilon_{n-i}.
  \end{equation}
  Next we compare the suffixes of the two words in (\ref{f9}). By the
  definition of $\tau(\omega)$, the left one ends with
  \[
  \varepsilon_{k+1}, \ldots, \varepsilon_n=\varepsilon_1,
  \ldots,\varepsilon_{n-k},
  \]
  while the right one ends with
  \[
  \varepsilon_{k-i+1}, \ldots, \varepsilon_{n-i}.
  \]
  By the self-admissibility of $\varepsilon_1,\cdots,\varepsilon_n$,
  we get
  \begin{equation}\label{fff8}
    \varepsilon_{k+1}, \ldots, \varepsilon_{n} = \varepsilon_1,
    \ldots, \varepsilon_{n-k}\succeq\varepsilon_{k-i+1}, \ldots,
    \varepsilon_{n-i}.
  \end{equation}
  Then the formula (\ref{f9}) and (\ref{fff8}) enable us to conclude the
  result.
\end{proof}

We give a sufficient condition ensuring a word being of full recurrence time.
\begin{lem}\label{l3} Assume that
$(\varepsilon_1, \ldots, \varepsilon_{m-1},\varepsilon_m)$ and
  $(\varepsilon_1, \ldots, \varepsilon_{m-1},
  \overline{\varepsilon}_m)$ are both self-admissible and $0\le
  \varepsilon_m<\overline{\varepsilon}_m$. Then
\[
\tau(\varepsilon_1, \ldots, \varepsilon_m) = m.
\]
\end{lem}
\begin{proof}
  Let $\tau(\varepsilon_1, \ldots, \varepsilon_m)=k.$ Suppose that
  $k<m$. We will show that this leads to a contradiction.  Write
  $m=tk+i$ with $0<i\le k$. By the definition of the recurrence time
  $\tau$, we have
  \begin{equation}\label{8}
    \sigma^{tk}(\varepsilon_1, \ldots, \varepsilon_m) =
    (\varepsilon_{tk+1}, \ldots, \varepsilon_{m}) = (\varepsilon_1,
    \ldots, \varepsilon_i).
  \end{equation}

  From the self-admissibility of the other sequence $(\varepsilon_1,
  \ldots, \varepsilon_{m-1},\overline{\varepsilon}_m)$, we know
  \begin{equation}\label{9}
    \sigma^{tk}(\varepsilon_1, \ldots, \varepsilon_{m-1},
    \overline{\varepsilon}_m) = (\varepsilon_{tk+1}, \ldots,
    \overline{\varepsilon}_{m})\preceq (\varepsilon_1, \ldots,
    \varepsilon_i).
  \end{equation}
The assumption $\varepsilon_m<\overline{\varepsilon}_m$ implies
  that
  \[
  (\varepsilon_{tk+1}, \ldots, \varepsilon_{m}) \prec
  (\varepsilon_{tk+1}, \ldots, \overline{\varepsilon}_{m}).
  \]
  Combining this with (\ref{8}) and (\ref{9}), we arrive at a
  contradiction $(\varepsilon_1, \ldots, \varepsilon_i) \prec
  (\varepsilon_1, \ldots, \varepsilon_i)$.
\end{proof}

\subsection{Maximal admissible sequence in parameter space}
%In this section, we would like to address the length of a cylinder in the parameter space (Definition \ref{cylinderpar}),
%which is closely related to the notion of the {\em recurrence time} defined in Section \ref{sec3.1}.

Now we recall a result of Schmeling \cite{Schme} concerning the length
of $I_n^P(\varepsilon_1, \ldots, \varepsilon_n)$.
\begin{lem}\cite{Schme}\label{l1}
  The cylinder $I_n^P(\varepsilon_1, \ldots, \varepsilon_n)$ is a
  half-open interval $[\beta_0, \beta_1)$. The left endpoint $\beta_0$
    is given as the only solution in $(1, \infty)$ of the equation
    \[
    1=\frac{\varepsilon_1}{\beta}+\cdots+\frac{\varepsilon_n}{\beta^n}.
    \]
    The right endpoint $\beta_1$ is the limit of the unique solutions
    $\{\beta_N\}_{N\ge n}$ in $(1,\infty)$ of the equations
    \[
    1 = \frac{\varepsilon_1}{\beta} + \cdots +
    \frac{\varepsilon_n}{\beta^n} +
    \frac{\varepsilon_{n+1}}{\beta^{n+1}} + \cdots +
    \frac{\varepsilon_N}{\beta^N}, \quad N\ge n
    \]
    where
    $(\varepsilon_1, \ldots, \varepsilon_n, \varepsilon_{n+1}, \ldots, \varepsilon_N)$
    is the maximal self-admissible sequence beginning with
    $\varepsilon_1, \ldots, \varepsilon_n$ in the lexicographical
    order. Moreover,
    \[\big|I_n^P(\varepsilon_1,\dots,\varepsilon_n)\big|\le
    {\beta_1^{-n+1}}.
    \]
\end{lem}

%\begin{rem}
%The maximal self-admissible word with length $m (m\ge n)$ in the above lemma is the prefix of that with length $m+1$.
%Then there exists a self-admissible sequence $(\varepsilon_k)_{k\ge 1}$ such that the right
%endpoint $\beta_1$ is the solution to the equation
%$$1=\frac{\varepsilon_1}{\beta}+\cdots+\frac{\varepsilon_n}{\beta^n}+
%\frac{\varepsilon_{n+1}}{\beta^{n+1}}+\cdots+\frac{\varepsilon_k}{\beta^k}+\cdots.$$
%\end{rem}
%\begin{rem}
%If the left endpoint of $I_n^P(\varepsilon_1,\cdots,\varepsilon_n)$ is 1,
%then the cylinder will be an open interval. For example, $I_2^P(1,0)=(1, \frac{1+\sqrt{5}}{2})$.
%\end{rem}

Therefore, to give an accurate estimate on the length of
$I_n^P(\varepsilon_1, \ldots, \varepsilon_n)$, we are led to determine
the maximal self-admissible sequence beginning with a given
self-admissible word $\varepsilon_1, \ldots, \varepsilon_n$.

\begin{lem}\label{l2}
  Let $w=(\varepsilon_1, \ldots, \varepsilon_n)$ be self-admissible
  with $\tau(w)=k$.  Then the periodic sequences
  \[
  (\varepsilon_1, \ldots, \varepsilon_k)^{m}\varepsilon_1, \ldots, \varepsilon_{\ell}, \ {\text{with}}\ 0\le \ell<k, \ km+\ell\ge n
  \]
  are the maximal self-admissible sequences beginning with
  $\varepsilon_1, \ldots, \varepsilon_n$. Consequently, if we denote by $\beta_1$ the right
  endpoint of $I_n^P(w_1, \ldots, w_n)$, then the $\beta_1$-expansion of 1 and the infinite $\beta_1$-expansion of 1 are given respectively as
   $$
   \varepsilon(1, \beta_1)=(\varepsilon_1, \ldots, \varepsilon_k+1),\ \  \varepsilon^*(1, \beta_1)=(\varepsilon_1, \ldots, \varepsilon_k)^{\infty}.
   $$
\end{lem}

\begin{proof} %We divide the proof into two cases according to $k=n$ or not.
By Lemma \ref{lem3.3},
%{\sc Case (I).}  $k=n$. This means that for each $1\le i<n$, \begin{equation}\label{f1}
%\sigma^i(\varepsilon_1,\cdots,\varepsilon_n)\ne \varepsilon_1,\cdots,\varepsilon_{n-i}.
%\end{equation} The self-admissibility of $(\varepsilon_1,\cdots,\varepsilon_{n})$ guarantees that for any $1\le i <n$
%\begin{equation}\label{f2}
%\sigma^i(\varepsilon_1,\cdots,\varepsilon_n)\preceq \varepsilon_1,\cdots,\varepsilon_{n-i}.
%\end{equation} Combining (\ref{f1}) and (\ref{f2}),
we get for all $1\le i<k$
\begin{equation}\label{f3}
  \varepsilon_{i+1},\ldots,\varepsilon_k\prec
  \varepsilon_1,\ldots,\varepsilon_{k-i}.
\end{equation}

For each $m\in \N$ and $0\le \ell<k$ with $km+\ell\ge n$, we check that $$w_0 = (\varepsilon_1, \ldots, \varepsilon_k)^{m}\varepsilon_1, \ldots, \varepsilon_{\ell}$$
is the maximal self-admissible sequence beginning with
$\varepsilon_1, \ldots, \varepsilon_{n}$ of order $mk+\ell$.

The admissibility of $w_0$ %By the periodicity of $w_0$, it suffices
%to check that for each $1\le i<k$,
%\[
% \varepsilon_{i+1}, \ldots, \varepsilon_k,(\varepsilon_1, \ldots,
% \varepsilon_k)^{\infty}\preceq (\varepsilon_1, \ldots,
% \varepsilon_k)^{\infty},
%\]
follows directly from (\ref{f3}).
Now we show the maximality of
$w_0$. Let
\[
w = (\varepsilon_1, \ldots, \varepsilon_k)^tw_1, \ldots, w_k, \ldots, w_{(m-t-1)k+1}, \ldots, w_{(m-t)k}, w_{(m-t)k+1}, \ldots, w_{(m-t)k+\ell}
\]
be a self-admissible word different from $w_0$, where $t\ge 1$ is the maximal integer such that $w$ begins with $(\varepsilon_1, \ldots, \varepsilon_k)^t$.
We distinguish two cases according to $t<m$ or $t=m$. We show for the case $t<m$ only since the other case can de done similarly.

If $t<m$, then
\[
w_{1}, \ldots, w_{k}\ne \varepsilon_1, \ldots, \varepsilon_k.
\]
The self-admissibility of $w$ ensures that
\[
w_{1}, \ldots, w_{k} \preceq \varepsilon_1, \ldots,
\varepsilon_k.
\]
Hence, we arrive at
\begin{equation}%\label{f4}
  w_{1}, \ldots, w_{k}\prec \varepsilon_1, \ldots,
  \varepsilon_k.
\end{equation}
This shows $w\prec w_0$.
%(ii). If $t=m$, then
%\[
%w_{1}, \ldots, w_{\ell}\ne \varepsilon_1, \ldots, \varepsilon_{\ell}.
%\]
%\vskip 3pt
%
%{\sc Case (II).}  $k<n$.
%By Lemma \ref{lem3.3}, we know that
%\begin{equation}\label{e3.8}
%\varepsilon_{i+1},\cdots,\varepsilon_k\prec\varepsilon_{1},\cdots,\varepsilon_{k-i}
%\end{equation}
%%Then by the minimality of $k$, it follows, similar to the first case, that for each $1\le i<k$,
%%\ \begin{equation}\label{f5}
%%\sigma^i(\varepsilon_1,\cdots,\varepsilon_n)< \varepsilon_1,\cdots,\varepsilon_{n-i}
%%\end{equation}
%for each $1\le i<k$.
%
%Now we show that $w_0=(\varepsilon_1,\cdots, \varepsilon_k)^{\infty}$ is
%the maximal self-admissible sequence beginning with $\varepsilon_1,\cdots,\varepsilon_{n}$.
%
%(1) Admissibility. By the periodicity of $w_0$, it suffices to check that for each $1\le i<k$,
%\begin{equation}\label{e3.9} \varepsilon_{i+1},\cdots,
%\varepsilon_k,(\varepsilon_1,\cdots, \varepsilon_k)^{\infty}\preceq (\varepsilon_1,\cdots, \varepsilon_k)^{\infty}.\end{equation}
%Note that the word in the left side begins with $\sigma^i(\varepsilon_1,\cdots,\varepsilon_n)$,
%while the word in the other side begins with $\varepsilon_1,\cdots,\varepsilon_{k-i}$.
%Thus, the inequality (\ref{e3.9}) follows from (\ref{e3.8}).
%%This gives the desired result immediately.
%
%(2) Maximality. Let $$w=(\varepsilon_1,\cdots,\varepsilon_k)^tw_1,\cdots,w_k, w_{k+1},
%\cdots,w_{2k}, \cdots, w_{ik+1},\cdots,w_{(i+1)k},\cdots$$ be a self-admissible word
%different from $w_0$. Then the left argument is completely the same as the case (I).
\end{proof}

From the proof of Lemma \ref{l2}, we have the following corollary. \begin{cor}\label{c3}
Assume that $(\varepsilon_1, \ldots, \varepsilon_m)$ is of full recurrence time. For any $1\le \ell\le k$ and $u_{\ell}=(w_{(\ell-1)m+1},\cdots, w_{\ell m})$ with $$
\sigma^i u_{\ell}\preceq \varepsilon_1, \ldots, \varepsilon_{m-i}, \ 0\le i<m,
$$ the word $$
(\varepsilon_1, \ldots, \varepsilon_m, u_1,\ldots,u_k)
$$ is self-admissible.
\end{cor}
%\begin{rem}
%The $\beta_1$-expansion of the number 1 is $(\varepsilon_1,\cdots,\varepsilon_{k-1},\varepsilon_k+1,0^\infty)$, but
%not $(\varepsilon_1,\cdots,\varepsilon_k)^\infty$.
%\end{rem}

The following simple calculation will be used several times in the
sequel, so we state it in advance.
\begin{lem}
  Let $1<\beta_0<\beta_1$ and $0\le \varepsilon_n< \beta_0$ for all
  $n\ge 1$. Then for every $n\ge 1$,
  \[
  \left(\frac{\varepsilon_1}{\beta_0}+\cdots+\frac{\varepsilon_n}{\beta_0^n}\right)-
  \left(\frac{\varepsilon_1}{\beta_1}+\cdots+\frac{\varepsilon_n}{\beta_1^n}\right)\le
  \frac{\beta_0}{(\beta_0-1)^2}(\beta_1-\beta_0).
  \]
\end{lem}
Now we apply Lemma \ref{l2} to give a lower bound of the length of
$I_n^P(\varepsilon_1, \ldots, \varepsilon_n)$.

\begin{thm}\label{t1}
  Let $w = (\varepsilon_1, \ldots, \varepsilon_n)$ be self-admissible
  with $\tau(w)=k$. Let $\beta_0$ and $\beta_1$ be the left and right
  endpoints of $I_n^P(\varepsilon_1, \ldots, \varepsilon_n)$. Then we
  have
  \begin{eqnarray*}
    \big|I_n^P(\varepsilon_1, \ldots, \varepsilon_n)\big| \ge \left\{
    \begin{array}{ll}
      C\beta_1^{-n}, & \hbox{{\rm{when}} k=n;}
      \\ C\frac{1}{\beta_1^{n}}\left(\frac{\varepsilon_{t+1}}{\beta_1}
      + \cdots + \frac{\varepsilon_k+1}{\beta_1^{k-t}}\right), &
      \hbox{{\rm{otherwise}}.}
    \end{array}
    \right.
  \end{eqnarray*}
  where $C:=(\beta_0-1)^2/\beta_0$ is a constant depending on
  $\beta_0$; the integers $t$ and $\ell$ are given as $\ell k <n\le
  (\ell+1)k$ and $t=n-\ell k$.
\end{thm}

\begin{proof}
  When $k=n$, the endpoints $\beta_0$ and $\beta_1$ of
  $I_n^P(\varepsilon_1, \ldots, \varepsilon_n)$ are given respectively
  as the solutions to
  \[
  1=\frac{\varepsilon_1}{\beta_0} + \cdots +
  \frac{\varepsilon_n}{\beta_0^n},
  \ {\text{and}}\ 1=\frac{\varepsilon_1}{\beta_1} + \cdots +
  \frac{\varepsilon_n+1}{\beta_1^n}.
  \]
  Thus,
  \begin{eqnarray*}
    &&\frac{1}{\beta_1^{n}} = \left(\frac{\varepsilon_1}{\beta_0} +
    \cdots + \frac{\varepsilon_n}{\beta_0^n}\right)-
    \left(\frac{\varepsilon_1}{\beta_1}+\cdots+\frac{\varepsilon_n}{\beta_1^n}\right)\le
    C^{-1}(\beta_1-\beta_0).
%=\sum_{i=1}^n\left(\frac{\varepsilon_i}{\beta_0^i}-\frac{\varepsilon_i}{\beta_1^i}\right)\\
%&&\ \ \ \ \  \le \sum_{i=1}^{\infty}\frac{i}{\beta_0^{i+1}}(\beta_1-\beta_0):
%=C_0(\beta_1-\beta_0), \ \ {\text{with}}\ C_0=\frac{1}{(\beta_0-1)^2}.
  \end{eqnarray*}
%where the inequality is because $$\frac{1}{\beta_0^i}-\frac{1}{\beta_1^i}=
%\frac{(\beta_1-\beta_0)(\beta_1^{i-1}+\beta_1^{i-2}\beta_0+
%\cdots+\beta_0^{i-1}))}{\beta_0^i\beta_1^i}\le \frac{i(\beta_1 - \beta_0)}{\beta_0^{i+1}}.$$
  Thus $|I_n^P(\varepsilon_1,\dots,\varepsilon_n)|=\beta_1 - \beta_0\ge C\beta_1^{-n}$.

%, where $C=C_0^{-1}$.

  When $k<n$, the endpoints $\beta_0$ and $\beta_1$ of
  $I_n^P(\varepsilon_1, \ldots, \varepsilon_n)$ are given respectively
  as the solutions to
  \[
  1 = \frac{\varepsilon_1}{\beta_0} + \cdots +
  \frac{\varepsilon_n}{\beta_0^n}, \quad {\text{and}} \quad 1 =
  \frac{\varepsilon_1}{\beta_1} + \cdots +
  \frac{\varepsilon_n}{\beta_1^n} +
  \frac{\varepsilon_{t+1}}{\beta_1^{n+1}} + \cdots +
  \frac{\varepsilon_k+1}{\beta_1^{(\ell+1)k}}.
  \]
  Thus,
  \begin{equation*}
    \frac{\varepsilon_{t+1}}{\beta_1^{n+1}} + \cdots +
    \frac{\varepsilon_k+1}{\beta_1^{(\ell+1)k}} =\left(\frac{\varepsilon_1}{\beta_0}
    + \cdots +
    \frac{\varepsilon_n}{\beta_0^n}\right)-\left(\frac{\varepsilon_1}{\beta_1}
    +\cdots+\frac{\varepsilon_n}{\beta_1^n}\right)\le
    C^{-1}(\beta_1-\beta_0),
%\le \sum_{i=1}^{\infty}\frac{i}{\beta_0^{i+1}}(\beta_1-\beta_0)\\&&=C_0(\beta_1-\beta_0).
  \end{equation*}
  and we obtain the desired result.
\end{proof}

Combining Lemma \ref{l1} and Theorem \ref{t1}, we know that when
$(\varepsilon_1, \ldots, \varepsilon_n)$ is of full recurrence time, the
length of $I_n^P(\varepsilon_1, \ldots, \varepsilon_n)$
satisfies
\[
C\beta_1^{-n}\le |I_n^P(\varepsilon_1, \ldots, \varepsilon_n)| \le
\beta_1^{-n}.
\]
In this case, $I_n^P(\varepsilon_1, \ldots, \varepsilon_n)$ is called a
regular cylinder.

\begin{rem}\label{r5}
  From Theorem \ref{t1}, if the digit $1$ appears regularly (i.e.\ the
  gap between two digits 1 is bounded) in a self-admissible sequence
  $w$, then we can have a good lower bound for the length of the
  cylinder generated by $w$. This will be applied in constructing a
  Cantor subset of $E(\{\ell_n\}_{n\ge 1}, x_0)$.
\end{rem}
%\begin{rem}
%As a consequence, we know that $$
%\big|I_n^P(\varepsilon_1,\cdots,\varepsilon_n)\big|\ge C\beta_1^{-2n+2}.
%$$  This is sharp for example $$(\varepsilon_1,\cdots, \varepsilon_n)=(2, 0,\cdots, 0, 2),$$ i.e,.
%$k=n-1$ and $\varepsilon_2=\cdots=\varepsilon_{n-1}=0$. Then $\beta_0, \beta_1$ are the solutions to $$
%1=\frac{2}{\beta_0}+\frac{2}{\beta_0^{n}}, \ {\text{and}}\ 1=\frac{2}{\beta_1}+\frac{2}{\beta_1^{n}}+
%\frac{1}{\beta_1^{2n-2}}.
%$$
%\end{rem}

\subsection{Distribution of regular cylinders}
%In this section, by applying the notion of {\em recurrence time}, we show
%a result concerning the distribution of cylinders with regular length. This result war presented
%for the first time by Persson and Schmeling \cite{PerS}, .

The following result presents a relationship between the recurrence
time of two consecutive cylinders in the parameter
space.
\begin{pro}\label{l4}
  Let $w_1,w_2$ be two self-admissible words of length $n$. Assume
  that $w_2\prec w_1$ and $w_2$ is next to $w_1$ in the lexicographic
  order. If $\tau(w_1)<n$, then
  \[
  \tau(w_2)>\tau(w_1).
  \]
\end{pro}

\begin{proof}
Since $\tau(w_1):=k_1<n$, $w_1$ can be written as
\[
w_1 = (\varepsilon_1, \ldots, \varepsilon_{k_1})^{t}, \varepsilon_1,
\ldots, \varepsilon_{\ell}, \ \ {\text{for some integers }}\ t\ge
1\ {\text{and}}\ 1\le \ell\le k_1.
\]
It is clear that $\varepsilon_1\ge 1$ which ensures the
self-admissibility of the sequence
\[
w = (\varepsilon_1, \ldots, \varepsilon_{k_1})^{t}, \underbrace{0,
  \ldots, 0}_{\ell}.
\]
Since $w_2$ is less than $w_1$ and is next to $w_1$, we have
\[
w\preceq w_2\prec w_1.
\]
This implies that $w_1$ and $w_2$ have common prefixes up to at least $k_1
\cdot t$ terms. Then $w_2$ can be expressed as
\[
w_2 = (\varepsilon_1, \ldots, \varepsilon_{k_1})^{t}, \varepsilon'_1,
\ldots, \varepsilon'_{\ell}.
\]

First, we claim that $\tau(w_2):=k_2\ne k_1$. Otherwise, by the
definition of $\tau(w_2)$, we obtain
\[
\varepsilon'_1, \ldots, \varepsilon'_{\ell} = \varepsilon_1, \ldots,
\varepsilon_{\ell},
\]
which indicates that $w_1=w_2$.

Second, we show that $k_2$ cannot be strictly smaller than
$k_1$. Otherwise, consider the prefix $\varepsilon_1, \ldots,
\varepsilon_{k_1}$ which is also the prefix of $w_1$. If $k_2<k_1$, we
have
\[
\varepsilon_{k_2+1}, \ldots, \varepsilon_{k_1} = \varepsilon_1,
\ldots, \varepsilon_{k_1-k_2},
\]
which contradicts Lemma \ref{lem3.3} by applying to $w_1$.

Therefore, $\tau(w_2)>\tau(w_1)$ holds.
\end{proof}

The following corollary indicates that cylinders with regular length (equivalent
with $\beta_1^{-n}$) are well distributed among the parameter space. This result
was found for the first time by Persson and Schmeling \cite{PerS}.
% but the ideas of the proof is different from ours.

\begin{cor}\label{c1}
  Among any $n$ consecutive cylinders in $C_n^P$, there is at least
  one with full recurrence time, hence with regular length.
\end{cor}

\begin{proof}
  Let $w_1\succ w_2\succ \cdots\succ w_n$ be $n$ consecutive cylinders
  in $C_n^P$. By Theorem~\ref{t1}, it suffices to show that there is
  at least one cylinder $w$ whose recurrence time is equal to $n$. If
  this is not the case, then by Proposition~\ref{l4}, we have
  \[
  1\le \tau(w_1)<\tau(w_2)<\cdots<\tau(w_n)<n,
  \]
  i.e.\ there would be $n$ different integers in $\{1, 2, \ldots,
  n-1\}$.  This is impossible. Thus we complete the proof.
\end{proof}

\section{Proof of Theorem \ref{maintheorem}: upper bound}
The upper bound of $\dim_\textsf{H}E(\{\ell_n\}_{n\ge 1}, x_0)$ is
given in a unified way no matter wheather $x_0=1$ or not.  Before
providing a upper bound of $\dim_\textsf{H}E(\{\ell_n\}_{n\ge 1},
x_0)$, we begin with a lemma.

%\begin{dfn}Let $(\varepsilon_1,\cdots, \varepsilon_n)$ be an admissible sequence. Define $$
%I_n^P(\varepsilon_1,\cdots, \varepsilon_n):=\{\beta: \varepsilon_i(1,\beta)=\varepsilon_i, 1\le i\le n\},
%$$which is called a fundamental set in the parameter space. \end{dfn}
%\begin{lem}\label{l1}Let $(\varepsilon_1,\cdots, \varepsilon_n)$ be an admissible sequence.
%Then the fundamental set $I_n^P(\varepsilon_1,\cdots, \varepsilon_n)$ is an interval.\end{lem}
%\proof Let $\beta_1,\beta_2\in I_n^P(\varepsilon_1,\cdots, \varepsilon_n)$.
%We show that for any $\beta\in (\beta_1,\beta_2)$, we still have $$
%\varepsilon_i(1,\beta)=\varepsilon_i, \ \forall \ 1\le i\le n.
%$$
%If we denote by $\pi_{\beta}$ as the $\beta$-expansion of 1, then it follows t
%hat $\pi_{\beta}$ is monotonic in the lexicographic order. Thus when $\beta\in (\beta_1,\beta_2)$, we have $$
%\pi_{\beta_1}<\pi_{\beta}<\pi_{\beta_2}.
%$$\hfill $\Box$

\begin{lem}\label{lll2}
  Let $(\varepsilon_1,\cdots, \varepsilon_n)$ be self-admissible. Then the set
  \begin{equation}\label{f8}
    \Big\{\, T^n_{\beta}1: \beta\in I_n^P(\varepsilon_1, \ldots,
    \varepsilon_n) \,\Big\}
  \end{equation}
  is a half-open interval $[0,a)$ for some $a\le 1$. Moreover,
    $T^n_{\beta}1$ is continuous and increasing on $\beta\in
    I_n^P(\varepsilon_1, \ldots, \varepsilon_n)$.
\end{lem}

\begin{proof}
  Note that for any $\beta\in I_n^P (\varepsilon_1, \ldots,
  \varepsilon_n)$, we have
  \[
  1=\frac{\varepsilon_1}{\beta}+\cdots+\frac{\varepsilon_n+T^n_{\beta}1}{\beta^n}.
  \]
  Thus
  \[
  T^n_{\beta}1 = \beta^n - \beta^n
  \left(\frac{\varepsilon_1}{\beta}+\cdots+\frac{\varepsilon_n}{\beta^n}\right).
  \]
  Denote
  \begin{equation}%\label{f4}
    f(\beta)=\beta^n-\Big(\varepsilon_1 \beta^{n-1}+\varepsilon_2
    \beta^{n-2}+\cdots+\varepsilon_n\Big).
  \end{equation}
  Then the set in \eqref{f8} is just
  \[
  \{\, f(\beta): \beta\in I_n^P (\varepsilon_1, \ldots, \varepsilon_n)
  \,\}.
  \]

  For the monotonicity of $T^n_{\beta}1$ on $\beta$, it suffices to
  show that the derivative $f'(\beta)$ is positive. In
  fact,
  \begin{eqnarray*}
    &&f'(\beta) = n\beta^{n-1}-\Big( (n-1)\varepsilon_1\beta^{n-2} +
    (n-2) \varepsilon_2\beta^{n-3} + \cdots+\varepsilon_{n-1}
    \Big)\\ &&\ \qquad \ge
    n\beta^{n-1}-(n-1)\beta^{n-1}\left(\frac{\varepsilon_1}{\beta}+
    \cdots+\frac{\varepsilon_{n-1}}{\beta^{n-1}}\right)\\ &&\ \qquad
    \ge n\beta^{n-1}-(n-1)\beta^{n-1} >0.
  \end{eqnarray*}

  Since $f$ is continuous and $I_n^P (\varepsilon_1, \ldots,
  \varepsilon_n)$ is an interval with the left endpoint $\beta_0$
  given as the solution to the equation
  \[
  1=\frac{\varepsilon_1}{\beta}+\cdots+\frac{\varepsilon_n}{\beta^n},
  \]
  the set \eqref{f8} is an interval with $0$ being its left endpoint
  and some right endpoint $a\le 1$.
\end{proof}

Now we estimate the upper bound of $\dim_\textsf{H} E
\big(\{\ell_n\}_{n\ge 1}, x_0\big)$.  For any $1<\beta_0<\beta_1$,
denote
\[
E(\beta_0,\beta_1) = \big\{\, \beta_0<\beta\le \beta_1:
|T^n_{\beta}1-x_0|<\beta^{-\ell_n}, \ {\text{i.o.}}\ n\in \mathbb{N}
\,\big\}.
\]
For any $\delta>0$, we partition the parameter space $(1,\infty)$ into
$\{(a_i, a_{i+1}]: i\ge 1\}$ with $\frac{\log a_{i+1}}{\log a_i} <
  1+\delta$ for all $i\ge 1$. Then
\[
E\big(\{\ell_n\}_{n\ge 1},
  x_0\big)=\cup_{i=1}^{\infty}E(a_i,a_{i+1}).
\]
By the $\sigma$-stability of the Hausdorff dimension, it suffices to
give a upper bound estimate on $ \dim_\textsf{H}E (\beta_0,\beta_1)$
for any $1<\beta_0<\beta_1$ with $\frac{\log \beta_1}{\log \beta_0} <
1+\delta$.
 %The following proposition implies that $\frac{1}{1+\alpha}$ is a upper bound of $\dim_\textsf{H} E(\{\ell_n\}_{n\ge 1}, x_0)$ by letting $\delta\to 0$.
\begin{pro} \label{pro:upperbound}
  For any $1<\beta_0<\beta_1$, we have
  \begin{equation}\label{upperbound}
    \dim_\textsf{H}E(\beta_0,\beta_1)\le \frac{1}{1+\alpha}\frac{\log
      \beta_1}{\log \beta_0}.
  \end{equation}
\end{pro}

\begin{proof}
  Let $B(x,r)$ be a ball with center $x\in [0, 1]$ and radius $r$. By
  using a simple inclusion $B(x_0, \beta^{-\ell_n})\subset B(x_0,
  \beta_0^{-\ell_n})$ for any $\beta>\beta_0$, we have
  \begin{eqnarray*}
    E(\beta_0, \beta_1)
    %&=&\{\beta\in(\beta_0, \beta_1): |T_\beta^n1-y|<\beta^{-\ell_n}\ \mbox{i.o.}\}\\
    %&=&\{\beta_0<\beta<\beta_1: T_\beta^n1\in B(y,\beta^{-l_n})\ \mbox{i.o.}\}\\
    &=& \bigcap_{N=1}^{\infty} \bigcup_{n=N}^{\infty} \left\{\,
    \beta\in(\beta_0, \beta_1]: T_\beta^n1\in B(x_0, \beta^{-\ell_n})
  \,\right\}\\ &\subset& \bigcap_{N=1}^{\infty}
  \bigcup_{n=N}^{\infty}\left\{\, \beta\in(\beta_0, \beta_1] :
T_\beta^n1\in B(x_0, \beta_0^{-\ell_n}) \,\right\}\\ &=&
\bigcap_{N=1}^{\infty} \bigcup_{n=N}^{\infty}\bigcup_{(i_1,\cdots,
  i_n) \in \Sigma_{\beta_0,\beta_1}^{P,n}} I_n^P(i_1, \ldots,
i_n;\beta_0^{-\ell_n}),
  \end{eqnarray*}
  where $\Sigma_{\beta_0,\beta_1}^{P,n}$ denotes the set of
  self-admissible words of length $n$ between
  $(\varepsilon_1^*(\beta_0),\allowbreak\dots,\allowbreak\varepsilon_n^*(\beta_0))$
  and $(\varepsilon_1^*(\beta_1),\dots,\varepsilon_n^*(\beta_1))$ in
  the lexicographic order, and
  \[
  I_n^P(i_1,\ldots, i_n;\beta_0^{-l_n}) := \{\beta\in(\beta_0,
  \beta_1]: \beta\in I_n^P(i_1, \ldots, i_n), T_\beta^n1\in B(x_0,
  \beta_0^{-\ell_n})\}.
  \]
  %where the last union  takes over all self-admissible words $(i_1,\cdots, i_n)$.

  By Lemma \ref{lll2}, we know that the set $I_n^P(i_1, \ldots,
  i_n;\beta_0^{-l_n})$ is an interval. In case it is non-empty we
  denote it's left and right endpoints by $ \beta_0'$ and $\beta_1' $
  respectively. Thus
  \[
  \beta_1'\le i_1+\frac{i_2}{\beta_1'}+\cdots
  +\frac{i_n}{\beta_1'^{n-1}}+\frac{x_0+\beta_0^{-\ell_n}}{\beta_1'^{n-1}}
  \]
  and
  \[
  \beta_0'\ge i_1 + \frac{i_2}{\beta_0'} +
  \cdots+\frac{i_n}{\beta_0'^{n-1}} +
  \frac{x_0-\beta_0^{-\ell_n}}{\beta_0'^{n-1}}\ge i_1 +
  \frac{i_2}{\beta_1'}+\cdots+\frac{i_n}{\beta_1'^{n-1}} +
  \frac{x_0-\beta_0^{-\ell_n}}{\beta_1'^{n-1}}.
  \]
  Therefore,
  \begin{align*}
    &\beta_1' - \beta_0' \\ \le&
    \left(i_1+\frac{i_2}{\beta_1'}+\cdots+\frac{i_n}{\beta_1'^{n-1}}+\frac{x_0+
      \beta_0^{-\ell_n}}{\beta_1'^{n-1}}\right) -
    \left(i_1+\frac{i_2}{\beta_1'}+\cdots+\frac{i_n}{\beta_1'^{n-1}}+
    \frac{x_0-\beta_0^{-\ell_n}}{\beta_1'^{n-1}}\right)\\ =&
    \frac{2\beta_0^{-\ell_n}}{\beta_1'^{n-1}}\le
    \frac{2\beta_0^{-\ell_n}}{\beta_0^{n-1}}=2\beta_0^{-(\ell_n+n-1)}.
  \end{align*}
  By the monotonicity of $\varepsilon(1,\beta)$ with respect to
  $\beta$ (Theorem \ref{t4} (2)), for any $\beta<\beta_1$,
  $\varepsilon(1,\beta)\in \Sigma_{\beta_1}$. Therefore
  \[
  \# \Sigma_{\beta_0,\beta_1}^{P,n}\le \# \Sigma_{\beta_1}^n\le
  \frac{\beta_1^{n+1}}{\beta_1-1},
  \]
  where the last inequality follows from Theorem \ref{Re}. It is clear
  that the family
  \[
  \Big\{\, I_n^P(i_1, \ldots, i_n,\beta_0^{-\ell_n}) : (i_1, \ldots,
  i_n) \in \Sigma_{\beta_0, \beta_1}^{P,n}, n\ge N\,\Big\}
  \]
  is a cover of the set $E(\beta_0,\beta_1)$. Recall that $\alpha =
  \liminf\limits_{n\to\infty}{\ell_n/n}$. Thus for any $s>
  \frac{1}{1+\alpha} \frac{\log \beta_1}{\log \beta_0}$, we have
  \begin{align*}
    \mathcal{H}^s(E(\beta_0,\beta_1))&\le \liminf_{N\to\infty}
    \sum_{n\ge N} \sum_{(i_1, \ldots,
      i_n) \in \Sigma_{\beta_0,\beta_1}^{P,n}} \big| I_n^P (i_1, \ldots,
    i_n,\beta_0^{-\ell_n}) \big|^s \\ &\le
    \liminf_{N\to\infty} \sum_{n\ge N} \frac{\beta_1^{n+1}}{\beta_1-1}
    \cdot 2^s\cdot \beta_0^{-(\ell_n+n-1)s}<\infty.
  \end{align*}
  This gives the estimate $\eqref{upperbound}$.
\end{proof}

\section{Lower bound of $E(\{\ell_n\}_{n\ge 1}, x_0)$: $x_0\ne 1$} \label{sec:lowerbound}

The proof of the lower bound of $\dim_\textsf{H}E\big(\{\ell_n\}_{n\ge
  1}, x_0\big)$, when $x_0\neq 1$, is done using a classic method:
first construct a Cantor subset $\F$, then define a measure $\mu$
supported on $\F$, and estimate the H\"{o}lder exponent of the measure
$\mu$. At last, conclude the result by applying the following
mass distribution principle \cite[Proposition 4.4]{Fal}.
\begin{pro}[Falconer \cite{Fal}]\label{p2}
  Let $E$ be a Borel subset of $\mathbb{R}^d$ and $\mu$ be a Borel
  measure with $\mu(E)>0$. Assume that, for any $x\in E$
  \[
  \liminf_{r\to 0}\frac{\log \mu(B(x,r))}{\log r}\ge s.
  \]
  Then $\dim_\textsf{H}E\geq s.$
\end{pro}

Instead of dealing with $E\big(\{\ell_n\}_{n\ge 1}, x_0\big)$
directly, we give some technical operation by considering the
following set
\[
E=\big\{\, \beta>1: |T^n_{\beta}1-x_0|<4(n+\ell_n)\beta^{-\ell_n},
\ {\text{i.o.}}\ n\in \mathbb{N} \,\big\}.
\]
It is clear that if we replace $\beta^{-\ell_n}$ by $\beta^{-(\ell_n+n
  \epsilon)}$ for any $\epsilon>0$ in defining $E$ above, the set $E$
will be a subset of $E(\{\ell_n\}_{n\ge 1}, x_0)$.
%Note that the shrinking speeds in the sets $E$ and $E(\{\ell_n\}_{n\ge 1}, x_0)$ have the same exponent $$
%\alpha=\liminf_{n\to\infty}\frac{\ell_n}{n}=\liminf_{n\to\infty}\frac{-\log 4(n+\ell_n)\beta^{-\ell_n}}{n\log \beta}.
%$$
Therefore, once we show the dimension of $E$ is bounded from below by
$1/(1+\alpha)$, so is $E\big(\{\ell_n\}_{n\ge 1}, x_0\big)$. Secondly, we always assume in the following that $\alpha>0$, if not, just replace $\ell_n$ by $\ell_n+n\epsilon$. The left
of this section is to prove that
\[
\dim_\textsf{H}E \ge \frac{1}{1+\alpha}, \ {\text{with}\ \alpha>0}.
\]

\subsection{Cantor subset}

Let $x_0$ be a real number in $[0,1)$. Let $\beta_0>1$ be such that
  its expansion $\varepsilon(1,\beta_0)$ of 1 is infinite, i.e.\ there
  are infinitely many nonzero terms in $\varepsilon(1,\beta_0)$. This
  infinity of $\varepsilon(1,\beta_0)$ implies that for each $n\ge 1$,
  the number $\beta_0$ is not the right endpoint of the cylinder
  $I_n^P(\beta_0)$ containing $\beta_0$ by Lemma \ref{l2}. Hence we
  can choose another $\beta_1>\beta_0$ such that the
  $\beta_1$-expansion $\varepsilon(1,\beta_1)$ of 1 is infinite and
  has a sufficiently long common prefix with $\varepsilon(1,\beta_0)$
  so that
\begin{equation}\label{condtion01}
  \frac{\beta_1(\beta_1-\beta_0)}{(\beta_0 - 1)^2} \le \frac{1-x_0}{2}.
\end{equation}
% Without loss of generality, we can assume that $\beta_0$
%and $\beta_1$ are as close as we want. The closeness of $\beta_0$ and $\beta_1$ depends on $x_0$.
%Write $\pi(\beta)$ for the digit sequence of the $\beta$-expansion of 1.

%Let $N_0$ be an integer such that the prefixes of $\pi(\beta_0)$ and $\pi(\beta_1)$ are different up to
%the $N_0$-th digit. Let $k$ be the recurrence time of the word $$
%\varepsilon_1(1,\beta_1),\cdots, \varepsilon_{N_0}(1,\beta_1).
%$$ Let $N=N(x_0)>>N_0$ be an integer sufficiently large and be a multiple of $k$ (the dependence on
%$x_0$ will be used in Lemma \ref{ll2} below). Then by the property on the maximal word with given prefix,
%we know that $$
%(\varepsilon_1(1,\beta_1),\cdots, \varepsilon_{N}(1,\beta_1))^{\infty}
%$$ is the largest number in the cylinder $$C_N^P:=C_N^P(\varepsilon_1(1,\beta_1),\cdots, \varepsilon_{N}(1,\beta_1)).$$

%Let $M$ be the integer such that the prefixes of $\pi(\beta_0)$ and $\pi(\beta_1)$ are same up to the $M$-th digit,

Let
\[
M = \min\{\, n\ge 1: \varepsilon_n(1,\beta_0)\neq
\varepsilon_n(1,\beta_1) \,\},
\]
that is, $\varepsilon_i(1,\beta_0)=\varepsilon_i(1,\beta_1)$ for all
$1\le i<M$ and $\varepsilon_M(1,\beta_0)\neq
\varepsilon_M(1,\beta_1)$.
%Denote $I_M^P:=I_M^P(\varepsilon_1(1,\beta_1),\cdots, \varepsilon_{M}(1,\beta_1)).$
%Let $\beta_2$ be an element in $(\beta_0,\beta_1)\setminus I_M^P$
%such that $S_{\beta_2}$ is a subshift of finite type. Such a choice on $\beta_2$ ensures that
Let $\beta_2$ be the maximal element beginning with
$w(\beta_0):=(\varepsilon_1(1,\beta_0),\cdots,
\varepsilon_{M}(1,\beta_0))$ in its infinite expansion of $1$, that is,
$\beta_2$ is the right endpoint of $I_M^P(w(\beta_0))$. Then it
follows that $\beta_0<\beta_2<\beta_1$. Note that the
word
\[
(\varepsilon_1(1,\beta_0), \ldots, \varepsilon_{M-1}(1,\beta_0),
\allowbreak\varepsilon_{M}(1,\beta_1)) = (\varepsilon_1(1,\beta_1),
\ldots, \varepsilon_{M-1}(1,\beta_1),
\allowbreak\varepsilon_{M}(1,\beta_1))
\]
is self-admissible and
$\varepsilon_M(1,\beta_0)< \varepsilon_M(1,\beta_1)$. So by Lemma
\ref{l3}, we know that $\tau(w(\beta_0))=M$. As a result, Lemma
\ref{l2} compels that the infinite $\beta_2$-expansion of $1$
is
\begin{equation}\label{fff6}
  \varepsilon^*(1,\beta_2) = (\varepsilon_1(1,\beta_0), \ldots,
  \varepsilon_{M}(1,\beta_0))^{\infty}.
\end{equation}

Since the following fact will be used frequently, we highlight it
here:
\begin{equation}\label{1}
  \varepsilon_1^*(1,\beta_2), \ldots, \varepsilon_{M}^*(1,\beta_2)\prec
  \varepsilon_1(1,\beta_1), \ldots, \varepsilon_{M}(1,\beta_1).
\end{equation}

\begin{lem}\label{ll1}
  For any $w\in S_{\beta_2}$, the sequence
  \[
  \varepsilon=\varepsilon_1(1,\beta_1), \ldots,
  \varepsilon_{M}(1,\beta_1), 0^M, w
  \]
  is self-admissible.
\end{lem}

\begin{proof}
  This will be checked by using properties of the recurrence time and
  the fact (\ref{1}). Denote $\tau(\varepsilon_1(1,\beta_1), \ldots,
  \varepsilon_{M}(1,\beta_1)) = k$. Then $\varepsilon_1(1,\beta_1),
  \ldots, \varepsilon_{M}(1,\beta_1)$ is periodic with a period
  $k$. Thus $\varepsilon$ can be rewritten as
  \[
  (\varepsilon_1, \ldots, \varepsilon_k)^{t_0} \varepsilon_1, \ldots,
  \varepsilon_s, 0^M, w
  \]
  for some $t_0\in \N$ and $0\le s<k$. We will compare
  $\sigma^i\varepsilon$ and $\varepsilon$ for all $i\ge 1$. The proof
  is divided into three steps according to $i\le M$, $M<i<2M$ or $i\ge
  2M$.

  (1) $i\le M$. When $i=tk$ for some $t\in \N$, then
  $\sigma^i(\varepsilon)$ and $\varepsilon$ have common prefix up to
  the $(M-tk)$-th digits. Following this prefix, the next $k$ digits
  in $\sigma^i(\varepsilon)$ is $0^k$, while that is $(\varepsilon_1,
  \ldots, \varepsilon_k)$ in $\varepsilon$, which implies
  $\sigma^i\varepsilon \prec \varepsilon$.

  When $i=tk+\ell$ for some $0<\ell<k$, then $\sigma^i(\varepsilon)$
  begins with $\varepsilon_{\ell+1}, \ldots, \varepsilon_s,0^{k-s}$ if
  $t=t_0$ and begins with $\varepsilon_{\ell+1}, \ldots,
  \varepsilon_k$ if $t=t_0$. By Lemma \ref{lem3.3}, we know that
  \[
  \varepsilon_{\ell+1}, \ldots,\varepsilon_s,0^{k-s} \preceq
  \varepsilon_{\ell+1}, \ldots,\varepsilon_k\prec \varepsilon_1,
  \ldots,\varepsilon_{k-\ell}.
  \]
  Thus $ \sigma^i(\varepsilon)\prec\varepsilon.  $

  (2) $M<i<2M$. For this case, it is trivial because
  $\sigma^i\varepsilon$ begins with $0$.

  (3) $i=2M+\ell$ for some $\ell\ge 0$. Then the sequence
  $\sigma^i(\varepsilon)$ begins with the subword $(w_{\ell+1},
  \ldots, w_{\ell+M})$ of $w$. Since $w\in S_{\beta_2}$, we have
  \[
  w_{\ell+1}, \ldots, w_{\ell+M}\preceq \varepsilon_1^*(1,\beta_2),
  \ldots, \varepsilon_{M}^*(1,\beta_2) \prec \varepsilon_1(1,\beta_1),
  \ldots, \varepsilon_{M}(1,\beta_1).
  \]
  where the last inequality follows from (\ref{1}). Therefore,
  $\sigma^i(\varepsilon)\prec\varepsilon$.
\end{proof}

Now we use Lemma \ref{lll2} and Lemma \ref{ll1}, and a suitable choice
of the self-admissible sequence, to show that the interval defined in
(\ref{f8}) can be large enough. Fix $q\ge M$ such that
\[
0^q\prec \varepsilon_{M+1}(1,\beta_1), \ldots,
\varepsilon_{M+q}(1,\beta_1),
\]
that is, find a position $M+q$ in $\varepsilon(1,\beta_1)$ with
nonzero element $\varepsilon_{M+q} (1,\beta_1)$. The choise of the
integer $q$ guarantees that the cylinder $I_{M+q}^P
(\varepsilon_1(1,\beta_1), \ldots, \varepsilon_M(1,\beta_1),0^q)$ lies
on the left hand side of $\beta_1$.
 %Let $$
%\pi(\beta_2)=(e_1,\cdots,e_M)^{\infty}:=e_1,\cdots,e_n,\cdots.
%$$
%We can assume $M\ge N_0$, otherwise, take longer period to $e_1,\cdots,e_M$.

\begin{lem}\label{ll2}
  Suppose $\beta_0$ and $\beta_1$ are close enough such that
  \eqref{condtion01} holds.
%\begin{equation}\label{condtion01}
% \frac{\beta_1-\beta_0}{(\beta_0 - 1)^2} \le \frac{1-x_0}{2}.
% \end{equation}
  For any $w\in \Sigma_{\beta_2}^{n-M-q}$ ending with $M$ zeros, the interval
  \[
  \Gamma_n=\Big\{\, T^n_{\beta}1: \beta\in
  I_n^P(\varepsilon_1(1,\beta_1), \ldots, \varepsilon_{M}(1,\beta_1),
  0^q, w) \,\Big\}
  \]
  contains $(x_0+1)/2$.
\end{lem}

\begin{proof}
  Recall $\varepsilon^*(1,\beta_2) = (\varepsilon_1(1,\beta_0),
  \ldots, \varepsilon_M(1,\beta_0))^{\infty} := (e_1, \ldots,
  e_M)^{\infty}$. Since $w$ ends with $M$ zeros, the sequence $(w,
  (e_1, \ldots, e_M)^{\infty})$ is in $S_{\beta_2}$.  Thus, the number
  $\beta^*$ for which
  \[
  \varepsilon(1,\beta^*)=\varepsilon_1(1,\beta_1), \ldots,
  \varepsilon_{M}(1,\beta_1), 0^q, w, e_1,e_2, \ldots
  \]
  belongs to $I_n^P(\varepsilon_1(1,\beta_1), \ldots,
  \varepsilon_{M}(1,\beta_1), 0^q, w)$ by Lemma \ref{ll1}. Note that
  $\beta^*\le \beta_1$ by the choice of $q$. For such a number
  $\beta^*$,
  \[
  T^n_{\beta^*}1 = \frac{e_1}{\beta^*} + \frac{e_2}{\beta^{*2}} +
  \cdots \ge \frac{e_1}{\beta_1} + \frac{e_2}{\beta_1^2} + \cdots.
  \]
  Note also that
  \[
  1=\frac{e_1}{\beta_2}+\frac{e_2}{\beta_2^2}+\cdots.
  \]
  Thus
  \[
  1-T_{\beta^*}^n1 \le
  \Big(\frac{e_1}{\beta_2}+\frac{e_2}{\beta_2^2}+\cdots\Big) -
  \Big(\frac{e_1}{\beta_1}+\frac{e_2}{\beta_1^2}+\cdots\Big) \le
  \frac{\beta_1(\beta_1-\beta_0)}{(\beta_0 - 1)^2}.
  \]
  Hence $T^n_{\beta^*}1 > \frac{x_0 + 1}{2}$ by
  \eqref{condtion01}. Then we obtain the statement of Lemma \ref{ll2}.
\end{proof}

Now we are in the position to construct a Cantor subset $\mathcal{F}$
of $E$.  Let $\mathfrak{N}$ be a subsequence of integers such that
\[
\liminf_{n\to \infty}\frac{\ell_n}{n}=\lim_{n\in \mathfrak{N},
  \ n\to\infty}\frac{\ell_{n}}{n}=\alpha>0.
\]

{\em Generation $0$ of the Cantor set}. Write
\[
\varepsilon^{(0)}=(\varepsilon_1(1,\beta_1), \ldots,
\varepsilon_{M}(1,\beta_1), 0^q),
\ {\text{and}}\ \mathbb{F}_0=\{\varepsilon^{(0)}\}.
\]
Then the $0$-th generation of the Cantor set is defined as
\[
\mathcal{F}_0=\Big\{\, I_{M+q}^P(\varepsilon^{(0)}):
\varepsilon^{(0)}\in \F_0 \,\Big\}.
\]

\medskip
{\em Generation $1$ of the Cantor set}. Recall that $M$ is the integer
defined for $\beta_2$ in the beginning of this subsection. Let $N\gg
M$. Denote by $U_\ell$ a collection of words in
$S_{\beta_2}$:
\begin{equation}\label{12}
  U_\ell=\Big\{\, u = (0^M, 1, 0^M, a_1, \ldots,a_{N}, 0^M, 1,0^M):
  (a_1, \ldots, a_{N})\in S_{\beta_2} \,\Big\},
\end{equation}
where $\ell=4M+2+N$ is the length of the words in $U_\ell$.
Without causing any confusion, in the sequel, the family $\F_0$ of
words is also called the 0-th generation of the Cantor set
$\mathcal{F}$.

\begin{rem}
  We give a remark on the way that the family $U_{\ell}$ is
  constructed.

  (1) The first $M$-zeros guarantee that for any $\beta_2$-admissible
  word $v$ and $u\in U_{\ell}$, the concatenation $(v, u)$ is still
  $\beta_2$-admissible.

  (2) With the same reason as (1), the other three blocks $0^M$ guarantee that $U_{\ell}\subset
  \Sigma_{\beta_2}^{\ell}$.

  (3) The two digits 1 are added to claim that the digit 1 appears
  regularly in $u\in U_{\ell}$ (recall  Remark \ref{r5} positioned after Theorem \ref{t1}).
\end{rem}

Let $m_{0}=M+q$ be the length of $\varepsilon^{(0)}\in \F_{0}$. Choose
an integer $n_1\in \mathfrak{N}$ such that $n_1\gg m_{0}$,
$\beta_0^{-{n_1}}\le 2(\beta_0-1)^2/\beta_1$ and
$4(n_1+\ell_{n_1})\beta^{-\ell_{n_1}} < \frac{1-x_0}{2}$ by noting $\alpha>0$. Write
\[
n_1-m_{0}=t_1\ell+i, \ \ {\text{for some}}\ t_1\in \N, \ 0\le i<\ell.
\]

First, we collect a family of self-admissible sequences beginning with
$\varepsilon^{(0)}$:
\[
\mathfrak{M}(\varepsilon^{(0)}) = \Big\{\, (\varepsilon^{(0)}, u_1,
\ldots, u_{t_1-1}, u_{t_1},0^i): u_1, \ldots, u_{t_1}\in U_\ell
\,\Big\}.
\]
Here the self-admissibility of the elements in
$\mathfrak{M}(\varepsilon^{(0)})$ follows from Lemma \ref{ll1}.

Second, for each $w\in \mathfrak{M}(\varepsilon^{(0)})$, we will
extract an element belonging to $\F_1$ (the first generation of
$\mathcal{F}$).  Let $\Gamma_{n_1}(w):=\{T_\beta^{n_1}1: \beta\in
I_{n_1}^P(w)\}.$ By Lemma \ref{ll2}, we have
that
\begin{equation}\label{6} \Gamma_{n_1}=\Gamma_{n_1}(w)\supset
  B(x_0, 4(n_1+\ell_{n_1})\beta_0^{-\ell_{n_1}}).
\end{equation}

Now we consider all possible self-admissible sequences of order
$n_1+\ell_{n_1}$ beginning with $w$, denoted by
\[
\A(w) := \big\{\, (w, \eta_1, \ldots, \eta_{\ell_{n_1}}): (w, \eta_1,
\ldots, \eta_{\ell_{n_1}}) \ \mbox{is self-admissible} \, \big\}.
\]
Then
\begin{equation}\label{7}
  \Gamma_{n_1}(w)= \bigcup_{\varepsilon\in \A(w)} \big\{\,
  T_{\beta}^{n_1}1 : \beta\in I_{n_1+\ell_{n_1}}^P(\varepsilon)
  \,\big\}.
\end{equation}

We show that for each $\varepsilon\in \A(w)$,
\begin{equation}\label{5}
  \big|\big\{\, T_{\beta}^{n_1}1:\beta\in
  I_{n_1+\ell_{n_1}}^P(\varepsilon) \,\big\}\big|\le
  4\beta_0^{-\ell_{n_1}}.
\end{equation}
In fact, for each pair $\beta,\beta'\in
I_{n_1+\ell_{n_1}}^P(\varepsilon)$, we have
\[
T^{n_1}_{\beta}1 = \frac{\eta_1}{\beta} + \cdots +
\frac{\eta_{\ell_{n_1}} + y}{\beta^{\ell_{n_1}}}, \qquad T^{n_1}_{\beta'}1
= \frac{\eta_1}{\beta'} + \cdots + \frac{\eta_{\ell_{n_1}} +
  y'}{\beta'^{\ell_{n_1}}}
\]
for some $0\le y, y'\le 1$. Then
\begin{eqnarray*}
  \Big|T^{n_1}_{\beta}1-T^{n_1}_{\beta'}1\Big| &\le&
  \sum_{k=1}^{\ell_{n_1}} \left|\frac{\eta_k}{\beta^k} -
  \frac{\eta_k}{\beta'^k}\right| + \frac{1}{\beta^{\ell_{n_1}}} +
  \frac{1}{\beta'^{\ell_{n_1}}}\\ &\le& \frac{\beta_1}{(\beta_0-1)^2}
  \beta_0^{-n_1 - \ell_{n_1}} + \frac{1}{\beta^{\ell_{n_1}}} +
  \frac{1}{\beta'^{\ell_{n_1}}} \le 4\beta_0^{-\ell_{n_1}}.
\end{eqnarray*}
%where the last inequality is from $\beta,\beta'\ge \beta_0$
%and $$\sum_{k=1}^\infty \frac{k}{\beta^k}|\beta - \beta'|\le
%\frac{1}{(\beta_0-1)^2}|\beta-\beta'|\le\frac{1}{(\beta_0-1)^2}\beta_0^{-(n_1+\ell_{n_1})}\le
%2\beta_0^{-\ell_{n_1}}.$$
Then Lemma \ref{lll2}, together with the estimate (\ref{5}), enables
us to conclude the following simple facts:
\begin{itemize}
\item for each $\varepsilon\in \A(w)$, $\big\{\,
  T_{\beta}^{n_1}1:\beta\in I^P_{n_1+\ell_{n_1}}(\varepsilon)
  \,\big\}$ is an interval, since $I^P_{n_1+\ell_{n_1}}(\varepsilon)$
  is an interval;

\item for every pair $\varepsilon, \varepsilon'\in \A(w)$, if
  $\varepsilon\prec\varepsilon'$, then by the monotonicity of
  $T_{\beta}^{n_1}1$ with respect to $\beta$ we have $\big\{\,
  T_{\beta}^{n_1}1:\beta\in I^P_{n_1+\ell_{n_1}}(\varepsilon)
  \,\big\}$ lies on the left hand side of
  $\big\{T_{\beta}^{n_1}1:\beta\in
  I^P_{n_1+\ell_{n_1}}(\varepsilon')\big\}$. Therefore, these
  intervals of the union in \eqref{7} are arranged in $[0,1]$
  consecutively;

\item moreover, there are no gaps between adjoint intervals in the
  union of \eqref{7}, since $\Gamma_{n_1}(w)$ is an interval;

\item the length of the interval $\big\{\, T_{\beta}^{n_1}1:\beta\in
  I^P_{n_1+\ell_{n_1}}(\varepsilon) \,\big\}$ is less than
  $4\beta_0^{-\ell_{n_1}}$.
\end{itemize}

By these four facts, we conclude that there are at least
$(n_1+\ell_{n_1})$ consecutive cylinders
$I_{n_1+\ell_{n_1}}^P(\varepsilon)$ with $\varepsilon\in \A(w)$ such
that $\big\{\, T_{\beta}^{n_1}1:\beta\in
I^P_{n_1+\ell_{n_1}}(\varepsilon) \,\big\}$ are contained in the ball
$B(x_0, 4(n_1+\ell_{n_1})\beta_0^{-\ell_{n_1}})$. Thus by Corollary
\ref{c1}, there exists a cylinder, denoted by
\[
I_{n_1+\ell_{n_1}}^P(w, w_1^{(1)}, \ldots, w_{\ell_{n_1}}^{(1)})
\]
satisfying that
\begin{itemize}
\item The recurrence time is full, i.e.\ $\tau(w, w_1^{(1)}, \ldots,
  w_{\ell_{n_1}}^{(1)})=n_1+\ell_{n_1}$;

\item The set $\big\{\, T_{\beta}^{n_1}1:\beta\in
  I_{n_1+\ell_{n_1}}^P (w, w_1^{(1)}, \ldots, w_{\ell_{n_1}}^{(1)})
  \,\big\}$ is contained in $B(x_0,
  4(n_1+\ell_{n_1})\beta_0^{-\ell_{n_1}})$. Thus, for any $\beta\in
  I_{n_1+\ell_{n_1}}^P(w, w_1^{(1)},\ldots,
  w_{\ell_{n_1}}^{(1)})$,
  \begin{equation}\label{2}
    \Big|T^{n_1}_{\beta}1-x_0\Big| <
    4(n_1+\ell_{n_1})\beta_0^{-\ell_{n_1}}.
  \end{equation}
\end{itemize}
This is the cylinder corresponding to $w\in \mathfrak{M}(\varepsilon^{(0)})$ we are looking for in
composing the first generation of the Cantor set.

Finally the first generation of the Cantor set is defined as
\[
\F_1 = \Big\{\, \varepsilon^{(1)}=(w, w_1^{(1)}, \ldots,
w_{\ell_{n_1}}^{(1)}) : w\in \mathfrak{M}(\varepsilon^{(0)}) \,
\Big\}, \qquad \mathcal{F}_1 = \bigcup_{\varepsilon^{(1)}\in
  \F_1}I_{n_1+\ell_{n_1}}^P(\varepsilon^{(1)}),
\]
where $w_1^{(1)}, \ldots, w_{\ell_{n_1}}^{(1)}$ depend on
$w\in\mathfrak{M}(\varepsilon^{(0)})$, but we do not show this
dependence in notation for simplicity. Let $m_1=n_1+\ell_{n_1}.$

\medskip
{\em From generation $k-1$ to generation $k$ of the Cantor set
  $\mathcal{F}$}. Assume that the $(k-1)$-th generation $\F_{k-1}$ has
been well defined, which is composed by a collection of words with
full recurrence time.
%(A cylinder is said to be of full recurrence time, if its recurrence time equals its order.)

To repeat the process of the construction of the Cantor set, we
present similar results as Lemma \ref{ll1} and Lemma \ref{ll2}.

\begin{lem}\label{l5}
  Let $\varepsilon^{(k-1)}\in \mathbb{F}_{k-1}$. Then for any $u\in
  S_{\beta_2}$ ending with $M$ zeros, the
  sequence $$(\varepsilon^{(k-1)}, u)$$ is self-admissible.
\end{lem}

\begin{proof}
  Let $1\le i<m_{k-1}$, where $m_{k-1}$ is the order of
  $\varepsilon^{(k-1)}$. Since $\varepsilon^{(k-1)}$ is of full
  recurrence time, an application of Lemma \ref{lem3.3} yields that
  \[
  \sigma^i(\varepsilon^{(k-1)}, u)\prec\varepsilon^{(k-1)}.
  \]

  Moreover, combining the assumption of $u\in S_{\beta_2}$ and
  \eqref{1}, we obtain that any block of $M$ consecutive digits in $u$
  is strictly less than the prefix of $\varepsilon^{(k-1)}$. In other
  words, when $m_{k-1}\le i\le m_{k-1}+|u|-M$, we have $
  \sigma^i(\varepsilon^{(k-1)}, u)\prec\varepsilon^{(k-1)}.  $

  At last, since $u$ ends with $M$ zeros, clearly when $i\ge
  m_{k-1}+|u|-M$, we obtain $ \sigma^i(\varepsilon^{(k-1)},
  u)\prec\varepsilon^{(k-1)}.  $
\end{proof}

\begin{lem}\label{ll4}
  For any $\varepsilon^{(k-1)}\in \F_{k-1}$ and $u\in S_{\beta_2}$
  ending with $M$ zeros, write $n=|\varepsilon^{(k-1)}|+|u|$. Then
  \[
  \Gamma_n=\Big\{\, T^n_{\beta}1: \beta\in I_n^P(\varepsilon^{(k-1)},
  u) \,\Big\}
  \]
  contains $(x_0+1)/2$.
\end{lem}
\begin{proof}
  With the same argument as Lemma \ref{l5}, we can prove that the
  sequence $(\varepsilon^{(k-1)}, u, (e_1,\cdots,e_M)^{\infty})$ is
  self-admissible. Then with the same argument as that in Lemma
  \ref{ll2}, we can conclude the assertion.
\end{proof}

Let $\varepsilon^{(k-1)}\in \F_{k-1}$ be a word of length
$m_{k-1}$. Choose an integer $n_k\in \mathfrak{N}$ such that $n_k\gg
m_{k-1}$. Write
\[
n_k-m_{k-1}=t_k\ell+i, \ \ {\text{for some}}\ 0\le
i<\ell.
\]

We collect a family of self-admissible sequences beginning with
$\varepsilon^{(k-1)}$:
\[
\mathfrak{M}(\varepsilon^{(k-1)})= \Big\{\, \varepsilon^{(k-1)},
u_1,\ldots,u_{t_k-1}, u_{t_k},0^i: u_1,\ldots,u_{t_k}\in U_\ell \,\Big\}.
\]
Here the self-admissibility of the elements in
$\mathfrak{M}(\varepsilon^{(k-1)})$ follows from Lemma \ref{l5}.
% the fact that $\varepsilon^{(k-1)}$ is of full
%recurrence time.
%This is the unique difference in the construction of $\mathbb{F}_k$ from that of $\mathbb{F}_1$.
%So, we give a detailed argument about the self-admissibility of the elements in $\mathfrak{M}(\varepsilon^{(k-1)})$ here.

%Once the self-admissibility of the elements in $\mathfrak{M}(\varepsilon^{(k-1)})$
%is proved, the following argument for the construction of $\F_k$ is absolutely the same as that for $\F_1$.

Then in the light of Lemma \ref{ll4}, the left argument for the
construction of $\F_k$ (the $k$-th generation of $\mathcal{F}$) is
absolutely the same as that for $\F_1$.

For each $w\in \mathfrak{M}(\varepsilon^{(k-1)})$, we can extract a
word of length $n_k+\ell_{n_k}$ with full recurrence time belonging to
$\F_k$, denoted by
\[
(w, w^{(k)}_1,\ldots,w^{(k)}_{\ell_{n_k}}).
\]

Then the $k$-th generation $\mathbb{F}_k$ is defined
as
\begin{equation}\label{3}
  \F_k=\Big\{\, \varepsilon^{(k)}=(w,
  w^{(k)}_1,\ldots,w^{(k)}_{\ell_{n_k}}): w\in
  \mathfrak{M}(\varepsilon^{(k-1)}),\ \varepsilon^{(k-1)}\in \F_{k-1}
  \, \Big\},
\end{equation}
and
\[
\mathcal{F}_k=\bigcup_{\varepsilon^{(k)}\in
  \mathbb{F}_{k}}I^P_{n_k+\ell_{n_k}}(\varepsilon^{(k)}).
\]
Note also that $w^{(k)}_1, \ldots, w^{(k)}_{\ell_{n_k}}$ depend on $w$
for each $w\in \mathfrak{M}(\varepsilon^{(k-1)})$.

Continue this procedure, we get a nested sequence
$\{\mathcal{F}_k\}_{k\ge 1}$ consisting of cylinders. Finally, the
desired Cantor set is defined as
\[
\mathcal{F}=\bigcap_{k=1}^{\infty}\bigcup_{\varepsilon^{(k)}\in
  \mathbb{F}_k}I_{|\varepsilon^{(k)}|}^P(\varepsilon^{(k)}) =
\bigcap_{k=1}^{\infty}\bigcup_{\varepsilon^{(k)}\in
  \mathbb{F}_k}I_{n_k+\ell_{n_k}}^P(\varepsilon^{(k)}).
\]

\begin{lem}
  $\mathcal{F}\subset E$.
\end{lem}

\begin{proof}
  This is clear by (\ref{2}).
\end{proof}

\subsection{Measure supported on $\mathcal{F}$}

Though $\F$ can only be viewed as a {\em locally} homogeneous Cantor
set, we define a measure {\em uniformly} distributed among
$\mathcal{F}$. This measure is defined along the cylinders with
non-empty intersection with $\mathcal{F}$. For any $\beta\in
\mathcal{F}$, let $\{I_n^P(\beta)\}_{n\ge 1}$ be the cylinders
containing $\beta$ and write
\[
\varepsilon(1,\beta)=(\varepsilon^{(k-1)}, u_1,\ldots, u_{t_k},
w_1^{(k)},\ldots, w_{\ell_{n_k}}^{(k)}, \ldots ),
\]
Here the last block $u_{t_k}$ contains the left zeros $0^i$ and the
order of $\varepsilon^{(k-1)}$ is $n_{k-1}+\ell_{n_{k-1}}$.

Now define
\[
\mu\big(I_{M+q}^P(\varepsilon^{(0)})\big)=1,
\]
and let
\[
\mu\big(I_{n_1}^P(\varepsilon^{(0)},
u_1,\ldots,u_{t_1})\big)=\left(\frac{1}{\sharp
  \Sigma_{\beta_2}^N}\right)^{t_1}.
\]
In other words, the measure is uniformly distributed among the
offsprings of the cylinder $I_{M+q}^P(\varepsilon^{(0)})$ with
nonempty intersection with $\mathcal{F}$.

Next for each $n_1<n\le n_1+\ell_{n_1}$, let
\[
\mu\big(I_{n}^P(\beta)\big)=\mu\big(I_{n_1}^P(\beta)\big).
\]

Assume that $\mu\big(I_{n_{k-1}+\ell_{n_{k-1}}}^P(\beta)\big)$,
i.e.\ $\mu\big(I_{n_{k-1}+\ell_{n_{k-1}}}^P(\varepsilon^{(k-1)})\big)$
has been defined.

(1) Define
\begin{equation}\label{ff2}
  \mu\big(I_{n_k}^P(\varepsilon^{(k-1)},u_1,\ldots,u_{t_k})\big) :=
  \left(\frac{1}{\sharp\Sigma_{\beta_2}^{N}}\right)^{t_k} \mu
  \big(I_{|\varepsilon^{(k-1)}|}^P(\varepsilon^{(k-1)})\big) =
  \left(\prod_{j=1}^k \big(\sharp \Sigma_{\beta_2}^{N}\big)^{t_j}
  \right)^{-1}.
\end{equation}

(2) When $n_{k-1}+\ell_{n_{k-1}}<n<n_k$, let
\[
\mu\big(I_n^P(\beta)\big)=\sum_{I^P_{n_k}(w)\in \mathcal{F}_k:
  I^P_{n_k}(w)\cap I_n^P(\beta)\ne
  \emptyset}\mu\big(I_{n_k}^P(w)\big).
\]
More precisely, when $n=n_{k-1}+\ell_{n_{k-1}}+t\ell$,
\begin{equation}\label{ff3}
  \mu\big(I_n^P(\beta)\big)=\prod_{j=1}^{k-1}\left(\frac{1}{\sharp
    \Sigma_{\beta_2}^{N}}\right)^{t_j}\cdot \left(\frac{1}{\sharp
    \Sigma_{\beta_2}^{N}}\right)^{t},
\end{equation}
and when $n=n_{k-1}+\ell_{n_{k-1}}+t\ell+i$ for some $i\ne 0$, we have
\begin{equation}\label{ff4}
  \mu\big(I_{n_{k-1}+\ell_{n_{k-1}}+t\ell}^P(\beta)\big)\ge
  \mu\big(I_n^P(\beta)\big)\ge \max \Big\{\,
  \mu\big(I_{n_{k-1}+\ell_{n_{k-1}}+(t+1)\ell}^P(\beta)\big),
  \ \mu\big(I_{n_k}^P(\beta)\big) \,\Big\}.
\end{equation}

(3) When $n_k<n\le n_k+\ell_{n_k}$, take
\begin{equation}\label{ff5}
  \mu\big(I_{n}^P(\beta)\big)=\mu\big(I_{n_k}^P(\beta)\big).
\end{equation}

\subsection{Lengths of cylinders}
Now we estimate the lengths of cylinders with non-empty intersection
with $\mathcal{F}$.

Let $(\varepsilon_1,\cdots,\varepsilon_n)$ be self-admissible such
that $I_n^P := I_n^P(\varepsilon_1, \ldots, \varepsilon_n)$ has
non-empty intersection with $\mathcal{F}$. Thus there exists $\beta\in
\mathcal{F}$ such that $I_n^P$ is just the cylinder containing
$\beta$. Let $n_k\le n<n_{k+1}$ for some $k\ge 1$. The estimate of the
length of $I_n^P$ is divided into two cases according to the range of
$n$.

(1) When $n_k\le n<n_{k}+\ell_{n_k}$. The length of $I_n^P$ is bounded
from below by the length of cylinders containing $\beta$ with order
$n_{k}+\ell_{n_k}+M$.

By the construction of $\mathcal{F}_k$, we know that
$\varepsilon(1,\beta)$ can be expressed as
\[
\varepsilon(1,\beta)=(\varepsilon^{(k)}, 0^M, 1,\ldots),
\]
which implies the self-admissibility of $(\varepsilon^{(k)}, 0^M, 1)$.
Then clearly $(\varepsilon^{(k)}, 0^M, 0)$ is self-admissible as
well. Then by Lemma \ref{l3}, we know that $(\varepsilon^{(k)}, 0^M,
0)$ is of full recurrence time. Thus,
\begin{equation}\label{ff6}
  \big|I_n^P\big|\ge \big|I_{n_k+\ell_{n_k}+M}^P(\beta)\big|\ge
  C\beta_1^{-(n_k+\ell_{n_k}+M+1)}:=C_1 \beta_1^{-(n_k+\ell_{n_k})}.
\end{equation}

(2) When $n_{k}+\ell_{n_k}\le n<n_{k+1}$. Let
$t=n-n_k-\ell_{n_k}$. Write $\varepsilon(1,\beta)$ as
\[
\varepsilon(1,\beta) = (\varepsilon^{(k)}, \eta_1,\ldots,
\eta_{t},\ldots)
\]
for some $(\eta_1,\cdots,\eta_{t})\in \Sigma_{\beta_2}^{t}$. Lemma
\ref{l5} tells us that
\[
(\varepsilon^{(k)}, \eta_1,\ldots,\eta_{t}, 0^M, 1, 0^M)
\]
is self-admissible. Then with the same argument as case (1), we obtain
\begin{equation}\label{ff7}
  \big|I_n^P\big|\ge \big|I_{n+M+1}^P(\varepsilon^{(k)},\eta_1,
  \ldots, \eta_{t}, 0^M, 0)\big|\ge
  C\beta_1^{-(n+M+1)}:=C_1\beta_1^{-n}.
\end{equation}

\subsection{Measure of balls}

%\proof
%For any $\varepsilon^{(k)}=w,w_1^{(k)},\cdots,w_{\ell_{n_k}}^{(k)}\in \mathcal{F}_k$,
%where $w\in\mathfrak{M}(\varepsilon^{(k-1)})$ for some $\varepsilon^{(k-1)}$.
%Denote $\tau(\varepsilon^{(k)})=m$. Note that $|\varepsilon^{(k)}|=n_k+\ell_{n_k}$.
%
%(i) $m=n_k+\ell_{n_k}$. Then $\big|I_{|\varepsilon^{(k)}|}^P(\varepsilon^{(k)})\big|
%\ge C\beta_1^{-(n_k+\ell_{n_k})}$ by Theorem \ref{t1}.
%
%(ii) $m\le n_k$. By the structure of $w$ and Theorem \ref{t1}, we have
%$$\big|I_{|\varepsilon^{(k)}|}^P(\varepsilon^{(k)})\big|\ge C\beta_1^{-(n_k+\ell_{n_k}+4M+2+N)}.$$
%
%{\bf (iii) $n_k < m< n_k+\ell_{n_k}.$ How to control this case since we have
%nothing information about the nonzero element in $w_1^{(k)},\cdots,w_{\ell_{n_k}}^{(k)}$?}
%
%{\color{blue}Remark: We can estimate the length of $I_n^P$ by another result:
%if $(\varepsilon_1,\cdots,\varepsilon_{n-1},1)$ and $(\varepsilon_1,\cdots,\varepsilon_{n-1},0)$
%are both self-admissible, then we must have $$
%\tau(\varepsilon_1,\cdots,\varepsilon_{n-1},1)=n+1.
%$$ By this result, since $U$ can be concatenated after $\varepsilon^{(k)}$,
%so we have $$\big|I_{|\varepsilon^{(k)}|}^P(\varepsilon^{(k)})\big|\ge C\beta_1^{-(n_k+\ell_{n_k}+4M+2+N)}.$$}
%
%Note that $\mu(I_{|\varepsilon^{(k)}|}^P(\varepsilon^{(k)}))=\beta_2^{t_1N+\cdots+t_kN}\asymp \beta_2^{n_1+\cdots+n_k},$
%we can obtain \eqref{4}.
%\hfill $\Box$

Now we consider the measure of arbitrary balls $B(\beta,r)$ with
$\beta\in \mathcal{F}$ and $r$ small enough.  Together with the
$\mu$-measure and the lengths of cylinders with non-empty intersection
with $\mathcal{F}$ given in the last two subsections, it follows
directly that
\begin{cor}\label{c2}
  For any $\beta\in \mathbb{F}$,
  \begin{equation}\label{4}
    \liminf_{n\to\infty}\frac{\log \mu\big(I_{n}^P(\beta)\big)}{\log
      |I_{n+1}^P(\beta)|}\ge \frac{1}{1+\alpha}\frac{\log
      \beta_2}{\log \beta_1}\frac{N}{\ell},
  \end{equation}
  where $N$ and $\ell$ are the integers in the definition of
  $U_{\ell}$ (see (\ref{12})).
\end{cor}
%To investigate the measure of an
%arbitrary ball, we cite a result from
%\cite{PerS}.\begin{lem}\label{ll3} For any $\beta>1$ and $\varepsilon >
%0$, there is a number $n =n(\beta,\varepsilon)$
% and a set $V^p \supset B(\beta, \varepsilon)$ such that $V^p$ can be written as a union of
%$n$ consecutive cylinders of length $n$ and $r>(\beta+r)^{-n-1}$.
%\end{lem}
%This gives that for any $B(\beta,r)$, there exists an integer $n$ such that $$
%\mu(B(\beta,r))\le n \mu(I^P_n), \ {\text{and}}, \ r>(\beta+r)^{-n-1}.
%$$

First, we refine the cylinders containing some $\beta\in
\mathcal{F}$ as follows. For each $\beta\in \mathcal{F}$ and $n\ge 1$,
define
\begin{eqnarray*}
  J_n(\beta)=\left\{
  \begin{array}{ll}
    I^P_{n_k+\ell_{n_k}}(\beta), & \hbox{when $n_k\le
      n<n_k+\ell_{n_k}$ for some $k\ge 1$;} \\ I_n^P(\beta), &
    \hbox{when $n_k+\ell_{n_k}\le n<n_{k+1}$ for some $k\ge 1$.}
  \end{array}
  \right.
\end{eqnarray*}
and call $J_n(\beta)$ the {\em basic interval} of order $n$ containing
$\beta$.

Now fix a ball $B(\beta,r)$ with $\beta\in \mathcal{F}$ and $r$
small. Let $n$ be the integer such that
\[
\big|J_{n+1}(\beta)\big|\le r<\big|J_n(\beta)\big|.
\]
Let $k$ be the integer such that $n_k\le n<n_{k+1}$. The difference on
the lengths of $J_{n+1}(\beta)$ and $J_n(\beta)$ (i.e., $
|J_{n+1}(\beta)|<|J_n(\beta)|$) yields that $$n_{k}+\ell_{n_k}\le
n<n_{k+1}.$$

Recall the definition of $\mu$. It should be noticed that
\[
\mu\big(J_n(\beta)\big)=\mu\big(I_n^P(\beta)\big), \ \mbox{for all}
\ n\in \N.
\]
Then all {\em basic intervals $J$} with the same order are of equal
$\mu$-measure. So, to bound the measure of the ball $B(\beta, r)$ from
above, it suffices to estimate the number of basic intervals with
non-empty intersection with the ball $B(\beta, r)$. We denote this
number by $\mathcal{N}$. Note that for $n_{k}+\ell_{n_k}\le
n<n_{k+1}$, all basic intervals are of length no less than
$C_1\beta_1^{-n}$. Since $r\le |J_n(\beta)|\le \beta_0^{-n}$, we have
\[
\mathcal{N}\le 2r/(C_1\beta_1^{-n})+2\le
2\beta_0^{-n}/(C_1\beta_1^{-n})+2\le C_2\beta_0^{-n}\beta_1^{n}.
\]
It follows that
\begin{equation}\label{ff8}
  \mu\big(B(\beta,r)\big)\le C_2\beta_0^{-n}\beta_1^{n} \cdot
  \mu\big(I_n^P(\beta)\big).
\end{equation}

Now we give a lower bound for $r$. When $n<n_{k+1}-1$, we have
\begin{equation}\label{ff9} r\ge
  \big|J_{n+1}(\beta)\big|=\big|I^P_{n+1}(\beta)\big|.
%\ge C_1\beta_1^{-(n+1)}.
\end{equation}
When $n=n_{k+1}-1$, we have
\begin{equation}\label{ff10} r\ge
  \big|J_{n+1}(\beta)\big|\ge C_1\beta_1^{-n_k-\ell_{n_k}}
%\ge C_1\beta_1^{-(n_{k+1}+\ell_{n_{k+1}})}.
\end{equation}
Thus, by the formula (\ref{ff8}) (\ref{ff9}) (\ref{ff10}) and
Corollary \ref{c2}, we have
\[
\liminf_{r\to 0}\frac{\log \mu(B(\beta, r))}{\log r}\ge
\left(\frac{\log \beta_0-\log \beta_1}{\log \beta_1}+\frac{\log
  \beta_2}{\log \beta_1}\frac{N}{\ell}\right)\frac{1}{1+\alpha}.
\]

Applying the mass distribution principle (Proposition \ref{p2}), we obtain
\[
\dim_\textsf{H}E \ge \left(\frac{\log \beta_0-\log \beta_1}{\log
\beta_1}+\frac{\log \beta_2}{\log \beta_1}\frac{N}{\ell}\right)\frac{1}{1+\alpha}.
\]
Letting $N\to \infty$ and then $\beta_1\to\beta_0$, we arrive at
\[
\dim_\textsf{H}E \ge \frac{1}{1+\alpha}.
\]

{\section{Lower bound of $E(\{\ell_n\}_{n\ge 1}, x_0)$: $x_0=1$}}

We still use the classic strategy to estimate the dimension of
$E(\{\ell_n\}_{n\ge 1}, 1)$ from below. In fact, we will show a little
stronger result: for any $\beta_0<\beta_1$, the Hausdorff dimension of
the set $E(\{\ell_n\}_{n\ge 1}, 1)\cap (\beta_0, \beta_1)$ is
$1/(1+\alpha)$.

The first step is devoted to constructing a Cantor subset
$\mathcal{F}$ of $E(\{\ell_n\}_{n\ge 1}, 1)$. We begin with some
notation.

As in the beginning of Section 5.1, we can require that $\beta_0$ and
$\beta_1$ are sufficiently close such that the common prefix
\[
(\varepsilon_1(1,\beta_1), \ldots, \varepsilon_{M-1}(1,\beta_1))
\]
of $\varepsilon(1,\beta_0)$ and $\varepsilon(1,\beta_1)$ contains
at least four nonzero terms. Assume that $\varepsilon(1,\beta_1)$
begins with the word $ {o}=(a_1, 0^{r_1-1}, a_2,\allowbreak 0^{r_2-1},
a_3, 0^{r_3-1}, a_4) $ with $a_i\ne 0$. Let
\[
\overline{o}=(0^{r_1}, 1, 0^{r_2}, 1, 0^{r_3}),
\ \ \overline{O}=(0^{r_1}, 1, 0^{r_2+1}).
\]
By the self-admissibility of ${o}$, it follows that if $a_1=1$,
then $\min\{r_2, r_3\}\ge r_1$.  So it is direct to check that for any
$i\ge 0$,
\begin{equation}\label{7.2}
  \sigma^i(\overline{o})\prec \varepsilon_1(1,\beta_1), \ldots,
  \varepsilon_{(r_1+r_2+r_3+2)-i}(1,\beta_1).
\end{equation}

Recall that $\beta_2$ is given in (\ref{fff6}).  Fix an integer
$\ell\gg M$. Define the collection
\[
U_\ell=\big\{u=(\overline{o}, \varepsilon_1, \ldots,
\varepsilon_{\ell-r_1-r_2-r_3-2-M}, 0^M) \in
\Sigma_{\beta_2}^{\ell}\big\}.
\]
Following the same argument as the case (3) in proving Lemma
\ref{ll1} and then by (\ref{1}), we have for any $u\in U_{\ell}$ and $i\ge
r_1+r_2+r_3+2$,
\begin{equation}\label{7.3}
  \sigma^i(u)%\prec \varepsilon_1(1,\beta_1), \ldots,
%  \varepsilon_{(r_1+r_2+r_3+2)-i}(1,\beta_1)
\prec
  (\varepsilon_1(1,\beta_1), \ldots, \varepsilon_{M}(1,\beta_1)).
\end{equation}
Combining (\ref{7.2}) and (\ref{7.3}), we get for any $u\in U_{\ell}$ and $i\ge 0$, \begin{equation}\label{7.6}
  \sigma^i(u)%\prec \varepsilon_1(1,\beta_1), \ldots,
%  \varepsilon_{(r_1+r_2+r_3+2)-i}(1,\beta_1)
\prec
  (\varepsilon_1(1,\beta_1), \ldots, \varepsilon_{M}(1,\beta_1)).
\end{equation}

Recall that $q$ is the integer such that
\[
(\varepsilon_{M+1}(1,\beta_1), \ldots,
\varepsilon_{M+q}(1,\beta_1))\ne 0^q.
\]
With the help of (\ref{7.6}), we present a result with
the same role as that of Lemma~\ref{ll1}.

\begin{lem}\label{ll3}
  Let $k\in \mathbb{N}$. For any $u_1, \ldots,u_k\in U_\ell$, the word
  \[
  \varepsilon=(\varepsilon_1(1,\beta_1), \ldots,
  \varepsilon_M(1,\beta_1), 0^q, u_1, u_2, \ldots, u_k)
  \]
  is of full recurrence time.
\end{lem}

\begin{proof}
  We check that $\sigma^i(\varepsilon)\prec \varepsilon$ for all $i\ge
  1$. When $i<M+q$, the argument is absolutely the same as that for
  $i<M+q$ in Lemma \ref{ll1}. When $i\ge M+q$, it follows by
  (\ref{7.6}).
\end{proof}

\subsection{Construction of the Cantor subset}
Now we return to the set
\[
E_0 := \big\{\, \beta_0<\beta<\beta_1: \big|T^n_{\beta}1-1| <
\beta^{-\ell_n}, \ {\text{i.o.}}\ n\in \N \,\big\}.
\]

We will use the following strategy to construct a Cantor subset of
$E_0$.\\ $\bullet$ {\em{\sc Strategy:} If the $\beta$-expansion of 1
  has a long periodic prefix with period $n$, then $T^n_{\beta}1$ and
  $1$ will be close enough.}

Let $\{n_k\}_{k\ge 1}$ be a subsequence of integers such that
\[
\lim_{k\to\infty}\frac{\ell_{n_k}}{n_k}=\liminf_{n\to
  \infty}\frac{\ell_n}{n}=\alpha, \ {\text{and}}\ \ n_{k+1}\gg n_k,
\ \mbox{for all} \ k\ge 1.
\]
%To a little abuse of the notation, we write $\ell_k$ for $\ell_{n_k}$, for all $k\ge 1$.  At the same time, by enlarging $\ell_k$
%by at most $\ell+2M$ if necessary, we can assume that for each $k\ge 1$, $\ell_k$ can be written as
%Denote $z_k=[\frac{\ell_{n_k}}{n_k}]$ and write
%\begin{equation}\label{zkjk}
%\ell_{n_k}=z_k n_k+ M+ j_k\ell, \ {\text{with}} \ \ell\le (j_k+1) \ell\le  n_k-2M.
%\end{equation}

{\em First generation $\mathcal{F}_1$ of the Cantor set $\mathcal{F}$.}

Let $\varepsilon^{(0)} = (\varepsilon_1(1,\beta_1), \ldots,
\varepsilon_M(1,\beta_1), 0^q)$ and $m_0=M+q$. Write $n_1=m_0+t_1\ell
+i_1$ for some $t_1\in \N$ and $0\le i_1<\ell$.  Now consider the
collection of self-admissible words of length $n_1$
\[
\mathfrak{M}(\varepsilon^{(0)}) = \big\{\, (\varepsilon^{(0)}, u_1,
\ldots, u_{t_1}, 0^{i_1}): u_1, \ldots, u_{t_1}\in U_\ell \,\big\}.
\]
Lemma \ref{ll3} says that all the elements in
$\mathfrak{M}(\varepsilon^{(0)})$ are of full recurrence time.

Enlarging $\ell_{n_1}$ by at most $m_0+\ell$ if necessary, the number
$\ell_{n_1}$ can be written as
\begin{equation}\label{7.4}
  \ell_{n_1}=z_1n_1+m_0+j_1\ell, \ {\text{with}}\ z_1\in \N, \ 0\le
  j_1<t_1.
\end{equation}
Corollary \ref{c3} convinces us that for any $(\varepsilon_1, \ldots,
\varepsilon_{n_1})\in \mathfrak{M}(\varepsilon^{(0)})$, the word
\begin{equation}\label{7.1}
  \varepsilon:=\Big(\big(\varepsilon_1, \ldots, \varepsilon_{n_1}\big),
  \big(\varepsilon_1, \ldots, \varepsilon_{n_1}\big)^{z_1},
  \big(\varepsilon^{(0)}, u_1, \ldots, u_{j_1}\big)\Big)
\end{equation}
is self-admissible. In other words, $\varepsilon$ is a periodic
self-admissible word with length $n_1+\ell_{n_1}$. We remark that the
suffix $\big(\varepsilon^{(0)}, u_1, \ldots, u_{j_1}\big)$ is the
prefix of $(\varepsilon_1, \ldots, \varepsilon_{n_1})$ but not chosen
freely.

Now consider the cylinder
\[
I_{n_1+\ell_{n_1}}^P := I^P_{n_1+\ell_{n_1}} \Big(\big(\varepsilon_1,
\ldots, \varepsilon_{n_1}\big)^{z_1+1}, \big(\varepsilon^{(0)}, u_1,
\ldots, u_{j_1}\big)\Big).
\]
It is clear that for each $\beta\in I_{n_1+\ell_{n_1}}^P$, the
$\beta$-expansion of $T^{n_1}_{\beta}1$ and that of $1$ coincide for
the first $\ell_{n_1}$ terms. So, we conclude that for any $\beta\in
I^P_{n_1+\ell_{n_1}}$,
\begin{equation}\label{oneinequality}
  \big|T_{\beta}^{n_1}1-1\big|<\beta^{-\ell_{n_1}}.
\end{equation}
Now we prolong the word in (\ref{7.1}) to a word of full
recurrence time.
%For this, we define a special word of length $\ell $ $$
%\overline{u}= (0^{r_1}, 1, 0^{r_2}, 0^{\ell-r_1-r_2-1}).
%$$
Still by Corollary {\ref{c3}}, we know that $(\varepsilon, u_{j_1+1})$ is self-admissible, which implies the admissibility of the word $$ (\varepsilon, 0^{r_1}, 1, 0^{r_2}, 1).
$$ So, by Lemma \ref{l3}, we obtain that the word
$ (\varepsilon, \overline{O})$ is of full recurrence time. Then
finally, the first generation $\mathcal{F}_1$ of the Cantor set
$\mathcal{F}$ is defined as
\begin{multline*}
  \mathcal{F}_1=\Big\{\,I^P_{(n_1+\ell_{n_1}+r_1+r_2+2)}
  \Big(\big(\varepsilon_1, \ldots, \varepsilon_{n_1}\big)^{z_1+1},
  \big(\varepsilon^{(0)}, u_1, \ldots, u_{j_1},
  \overline{O}\big)\Big): \\ (\varepsilon_1,
  \ldots,\varepsilon_{n_1})\in \mathfrak{M}(\varepsilon^{(0)})
  \,\Big\}.
\end{multline*}

{\em Second generation $\mathcal{F}_2$ of the Cantor set
  $\mathcal{F}$.}

Let $m_1=n_1+\ell_1+r_1+r_2+2$ and
write $$n_2=m_1+t_2\ell+i_2\ {\text{ for some}}\ t_2\in \N, \ 0\le
i_2<\ell.$$ For each $\varepsilon^{(1)}\in \mathcal{F}_1$, consider
the collection of self-admissible words of length $n_2$
\[
\mathfrak{M}(\varepsilon^{(1)}) = \big\{\, (\varepsilon^{(1)}, u_1,
\ldots, u_{t_2}, 0^{i_2}): u_1, \ldots, u_{t_2}\in U_\ell \,\big\}.
\]
By noting that $\varepsilon^{(1)}$ is of full recurrence time and
by the formula (\ref{7.6}), we know that all elements in
$\mathfrak{M}(\varepsilon^{(1)})$ are of full recurrence time.

Similar to the modification on $\ell_{n_1}$, by enlarging $\ell_{n_2}$
by at most $m_1+\ell$ if necessary, the number $\ell_{n_2}$ can be
written as
\begin{equation}\label{7.5}
  \ell_{n_2}=z_2n_2+m_1+j_2\ell, \ {\text{with}}\ z_2\in \N, \ 0\le
  j_2<t_2.
\end{equation}
Then follow the same line as the construction for the first
generation, we get the second generation
$\mathcal{F}_2$,
\begin{multline*}
  \mathcal{F}_2 = \Big\{\, I^P_{(n_2+\ell_{n_2}+r_1+r_2+1)}
  \Big(\big(\varepsilon_1, \ldots,\varepsilon_{n_2}\big)^{z_2+1},
  \big(\varepsilon^{(1)}, u_1, \ldots, u_{j_2},
  \overline{O}\big)\Big): \\ (\varepsilon_1,
  \ldots,\varepsilon_{n_2})\in \mathfrak{M}(\varepsilon^{(1)})
  \,\Big\} .
\end{multline*}
We remark that the suffix $\big(\varepsilon^{(1)}, u_1, \ldots,
u_{j_2}\big)$ is the prefix of
$(\varepsilon_1,\cdots,\varepsilon_{n_2})$ but not chosen freely. Then
let $m_2=n_2+\ell_{n_2}+r_1+r_2+2$.

Then, proceeding along the same line, we get a nested sequence
$\mathcal{F}_k$ consisting of a family of cylinders. The desired
Cantor set is defined as
\[
\mathcal{F}=\bigcap_{k\ge 1}\mathcal{F}_k.
\]
Noting \eqref{oneinequality}, we know that $\mathcal{F} \subset E_0.$

\subsection{Estimate on the supported measure}

The remaining argument for the dimension of $\mathcal{F}$ is almost
the same as what we did in Section 5: constructing an evenly
distributed measure supported on $\mathcal{F}$ and then applying the
mass distribution principle. Thus, we
will not repeat it here.
%So we complete the proof for $x_0=1$.

\section{Proof of Theorem \ref{t2}}

The proof of Theorem \ref{t2} can be established with almost the same
argument as that for Theorem \ref{maintheorem}. Therefore only
differences of the proof are marked below.

\subsection{Proof of the upper bound}

For each self-admissible sequence $(i_1,\ldots,i_n)$,
denote
\[
J_n (i_1, \ldots, i_n):=\Big\{\, \beta\in I_n^P (i_1, \ldots, i_n):
|T_\beta^n1-x(\beta)|< \beta_0^{-\ell_n} \,\Big\}.
\]
These sets correspond to the sets $I_n^P (i_1, \ldots, i_n;
\beta_0^{-\ell_n})$ studied in the proof of
Proposition~\ref{pro:upperbound}, where the upper bound for the case
of constant $x_0$ was obtained. We have that
\[
\left(\widetilde{E}\big(\{\ell_n\}_{n\ge 1}, x\big)\cap
(\beta_0,\beta_1)\right)\subset
\bigcap_{N=1}^{\infty}\bigcup_{n=N}^{\infty}\ \ \bigcup_{(i_1, \ldots,
  i_n) \ {\text{self-admissible}}} J_n (i_1, \ldots, i_n).
\]
What remains is to estimate the diameter of $J_n (i_1, \ldots, i_n)$
for any self-admissible sequence $(i_1, \ldots, i_n)$. If we can get a
good estimate of the diameter, then we can do as in the proof of
Proposition~\ref{pro:upperbound} to get an upper bound of the
dimension of $\widetilde{E}\big(\{\ell_n\}_{n\ge 1}, x\big)\cap
(\beta_0,\beta_1)$.

Suppose $J_n$ is non-empty, and let $\beta_2 < \beta_3$ denote the
infimum and supremum of $J_n$. Let $L$ be such that $\beta \mapsto x
(\beta)$ is Lipschitz continuous, with constant $L$. Denote by $\psi$
the map $\beta \mapsto T_\beta^n (1)$, and note that $\psi$ satisfies
\[
|\psi (\beta_3) - \psi (\beta_2)| \geq \beta_0^n\cdot |\beta_3 - \beta_2|.
\]

Clearly, $\beta_2$ and $\beta_3$ must satisfy
\[
|\psi (\beta_3) - \psi(\beta_2)| - |x (\beta_3) - x(\beta_2)| < 2
\beta_0^{-\ell_n},
\]
and hence, we must have
\begin{equation} \label{eq:conditiononJ}
  \beta_0^n \cdot|\beta_3 - \beta_2| - L\cdot |\beta_3 - \beta_2| < 2
  \beta_0^{-\ell_n}.
\end{equation}
Take $K > 2$. Then we must have $|\beta_3 - \beta_2| \leq K
\beta_0^{-\ell_n - n}$ for sufficiently large $n$, otherwise
\eqref{eq:conditiononJ} will not be satisfied.

Thus, we have proved that $|J_n (i_1,\ldots,i_n)| \leq K \beta_0^{-
  \ell_n - n}$ for some constant $K$. This is all what is needed to
make the proof of Proposition~\ref{pro:upperbound} work also for the
case of non-constant $x_0$.

\subsection{Proof of the lower bound}

Case 1. If $x(\beta)=1$ for all $\beta\in [\beta_0, \beta_1]$, this
falls into the proof of Theorem \ref{maintheorem}.
\medskip

Case 2. Otherwise, we can find a subinterval of $(\beta_0,\beta_1)$
such that the supremum of $x(\beta)$ on this subinterval is strictly
less than 1. We denote by $0\le x_0<1$ the supremum of $x(\beta)$ on
this subinterval. We note that with this definition of $x_0$,
Lemma~\ref{ll2} still holds.

Now that we have Lemma~\ref{ll2}, we can get a lower bound in the same
way as in Section~\ref{sec:lowerbound}, i.e.\ we construct a Cantor
set with desired properties. The proof is more or less unchanged, but
some minor changes are nessesary, as we will describe below.

The sets $\mathbb{F}_0$ and $\mathfrak{M} (\varepsilon^{(0)})$ are
defined as before, and we consider a $w \in \mathfrak{M}
(\varepsilon^{(0)})$. On the interval $I_{n_1}^P (w)$ we define $\psi
\colon \beta \mapsto T_\beta^{n_1} (1)$, and we observe that there are
constants $c_1$ and $c_2$ such that
\[
c_1 \beta_0^{n_1} \leq \psi' (\beta) \leq c_2 \beta_0^{n_1},
\]
holds for all $\beta \in I_{n_1}^P (w)$. As in the proof of the upper
bound, we let $L$ denote the Lipschitz constant of the function $\beta
\mapsto x(\beta)$.

We need to estimate the size of the set
\[
J = \{\, \beta \in I_{n_1}^P (w) : \psi (\beta) \in B (x_0(\beta),
C (n_1 + \ell_{n_1}) \beta_0^{-\ell_{n_1}}) \,\}.
\]
The constant $C$ appearing in the definition of $J$ above, was equal
to $4$ in Section~\ref{sec:lowerbound}. We remark that the value of
$C$ has no influence on the result of the proof, so we may choose it
more freely, as will be done here.

Lemma~\ref{ll2} implies that there is a $\beta_a \in J$ such that
$\psi (\beta_a) = x (\beta_a)$. Suppose $\beta_b \in I_{n_1}^P (w)$ is
such that $|\beta_a - \beta_b| < 4 (n_1 + \ell_{n_1}) \beta_0^{-n_1 -
  \ell_{n_1}}$. We can choose
$C$ so large that we have
\begin{align*}
|\psi (\beta_b) - x (\beta_b)| & \leq |\psi (\beta_a) - \psi(\beta_b)|
+ |x (\beta_a) - x (\beta_b)| \\ & \leq c_2 4 (n_1 + \ell_{n_1})
\beta_0^{-\ell_{n_1}} + L\cdot 4 (n_1 + \ell_{n_1})
\beta_0^{- n_1 - \ell_{n_1}} < C (n_1 + \ell_{n_1})
\beta_0^{-\ell_{n_1}}.
\end{align*}
This proves that $\beta_b$ is in $J$, and hence, $J$ contains an
interval of length at least $4 (n_1 + \ell_{n_1}) \beta_0^{-n_1 -
  \ell_{n_1}}$.

Analogous to the estimate in \eqref{5}, we have that
$|I_{n_1+\ell_{n_1}}^P (\varepsilon)| \leq 4 \beta_0^{-n_1 -
  \ell_{n_1}}$. This implies that there are at least $(n_1 +
\ell_{n_1})$ consequtive cylinders $I_{n_1 + \ell_{n_1}}^P
(\varepsilon)$ with the desired hitting property, where $\varepsilon
\in \mathbb{A} (w)$.

With the changes indicated above, the proof then continues just as in
Section~\ref{sec:lowerbound}.

\section{Application}

This section is devoted to an application of Theorem
\ref{maintheorem}. For each $n\ge 1$, denote by $\ell_n(\beta)$ the
number of the longest consecutive zeros just after the $n$-th digit in
the $\beta$-expansion of 1, namely,
\[
\ell_n(\beta) := \max\{\, k\ge 0: \varepsilon_{n+1}^*(\beta) = \cdots
= \varepsilon_{n+k}^*(\beta)=0 \,\}.
\]
Let
\[
\ell(\beta)=\limsup_{n\to\infty}\frac{\ell_n(\beta)}{n}.
\]
Li and Wu \cite{LiW} gave a kind of classification of betas according
to the growth of $\{\ell_n\}_{n\ge 1}$ as follows:
\begin{align*}
  &A_0=\Big\{\,\beta>1: \{\ell_n(\beta)\}\ \text{is
    bounded}\,\Big\};\\ & A_1=\Big\{\,\beta>1:
  \{\ell_n(\beta)\}\ \mbox{is unbounded and}\ \ \ell(\beta)=0
  \,\Big\};\\ & A_2=\Big\{\,\beta>1: \ell(\beta)>0\,\Big\}.
\end{align*}

We will use the dimensional result of $E(\{\ell_n\}_{n\ge 1}, x_0)$ to
determine the size of $A_1, A_2$ and $A_3$ in the sense of the
Lebesgue measure $\mathcal{L}$ and Hausdorff dimension. In the
argument below only the dimension of $E(\{\ell_n\}_{n\ge 1}, x_0)$
when $x_0=0$ is used. In other words, the result in \cite{PerS} by
Persson and Schmeling is already sufficient for the following
conclusions.

\begin{pro}[Size of $A_0$]
$\mathcal{L}(A_0)=0$ and $\dim_\textsf{H}(A_0)=1$.
\end{pro}

\begin{proof}
  The set $A_0$ is nothing but the collections of $\beta$ with
  specification properties. Then this proposition is just Theorem A in
  \cite{Schme}.
\end{proof}

%\begin{rem}
% Since
%$A_0=\{\beta>1: \beta \mbox{ satisfies the specification }\},$
%by the result of Schmeling \cite{Schme}, we know that
%$$\mathcal{L}(A_0)=0\ \text{and}\ \dim_\textsf{H}(A_0)=1.$$
%\end{rem}

%As an application of Theorem \ref{maintheorem}, more precisely, the result due to Person and Schmeling
%\cite{PerS} is enough for this application, we can calculate the size of $A_1$ and $A_2$ by the senses of Lebesgue measure
%$\mathcal{L}$ and Hausdorff dimension.

\begin{pro}[Size of $A_2$]\label{dimlevelbeta}
$\mathcal{L}(A_2)=0$ and $\dim_\textsf{H}(A_2)=1$.
\end{pro}

\begin{proof}
  For any $\alpha>0$, let
  \[
  F(\alpha) = \left\{\, \beta>1: \ell(\beta)\geq\alpha \,\right\}.
  \]
  Then $A_2=\bigcup_{\alpha>0}F(\alpha)$. Since $F(\alpha)$ is
  increasing with respect to $\alpha$, the above union can be
  expressed as a countable union. Now we show that for each $\alpha>0$
  \[
  \dim_\textsf{H}F(\alpha)=\frac{1}{1+\alpha},
  \]
  which is sufficient for the desired result.

  Recall the algorithm of $T_{\beta}$. Since for each $\beta\in A_2$,
  the $\beta$-expansion of 1 is infinite, then for each $n\ge 1$, we
  have
  \[
  T^n_{\beta}1 = \frac{\varepsilon_{n+1}^*(\beta)}{\beta} +
  \frac{\varepsilon_{n+2}^*(\beta)}{\beta^2} + \cdots.
  \]
  Then by the definition of $\ell_n(\beta)$, it follows
% we know that $\varepsilon_{n+1}^*(\beta)=\cdots=\varepsilon_{n+\ell_n(\beta)}^*(\beta)=0$ and $\varepsilon_{n+\ell_n(\beta)+1}^*(\beta)\neq0$, then
  \begin{equation}\label{orbit1lnbeta}
    \beta^{-(\ell_n(\beta)+1)} \le T_\beta^n1
    %=(\varepsilon_{n+\ell_n(\beta)+1}^*(\beta)+T_\beta^{n+\ell_n(\beta)+1}1)\beta^{-(\ell_n(\beta)+1)}
    \le (\beta+1)\beta^{-(\ell_n(\beta)+1)}.
  \end{equation}
  As a consequence, for any $\delta>0$,
%Upper bound: For any $\delta>0$
%and any $\beta\in F(\alpha)$, there exist infinite many $n$ such that $\ell_n(\beta)>(\alpha-\delta)n$, which implies
%$T_\beta^n1<(\beta+1)\beta^{-n(\alpha-\delta)-1}$ together with the right inequality of \eqref{orbit1lnbeta}. Then
  \begin{equation}\label{8.1}
    F(\alpha)\subset \{\,\beta>1:
    T_\beta^n1<(\beta+1)\beta^{-n(\alpha-\delta)-1}\ {\text{for
        infinitely many}}\ n\in \mathbb{N}\,\}.
  \end{equation}
%Applying Theorem \ref{maintheorem}, we obtain that  $$\dim_\textsf{H}F(\alpha)\le \frac{1}{1+\alpha-\delta}.$$
%Let $\delta\to 0$, we get the desired upper bound of $\dim_\textsf{H}F(\alpha)$.

  On the other hand, it is clear that
%Lower bound: We claim that
  \begin{equation}\label{claimlowerdim}
    \{\,\beta>1: T_\beta^n1<\beta^{-n\alpha}\ {\text{for infinitely
        many}}\ n\in \mathbb{N}\,\}\subset F(\alpha).
  \end{equation}
%In fact, by the left inequality of \eqref{orbit1lnbeta}, we know that $T_\beta^n1<\beta^{-n\alpha}$
%implies $\frac{\ell_n(\beta)}{n}>\alpha-\frac{1}{n}$, which indicates that \eqref{claimlowerdim} holds.
  Applying Theorem \ref{maintheorem} to (\ref{8.1}) and
  \eqref{claimlowerdim}, we get that
  \[
  \dim_\textsf{H}F(\alpha)= \frac{1}{1+\alpha}. \qedhere
  \]
\end{proof}

Since $A_1=(1,\infty)\setminus (A_0\cup A_2)$, it follows directly
that
\begin{pro}[Size of $A_1$] %$\mathcal{L}(A_1)=1$.
%(1). $\mathcal{L}(A_1)=1$ and $\dim_{\rm H}A_2=1$.
%
%(2). For $\mathcal{L}$-almost every $\beta>1$,
%$$\lim_{n\to\infty}\frac{\ell_n(\beta)}{n}=0.$$
The set $A_1$ is of full Lebesgue measure.
\end{pro}
%\proof
%Let $m\ge 1$ and denote
%$$G(m)=\left\{\beta>1: \ell(\beta)>\frac{1}{m} \right\}.$$
% From Corollary \ref{dimlevelbeta}, we know that $\dim_\textsf{H}G(m)=\frac{1}{1+\frac{1}{m}}<1$ and
%$\mathcal{L}(G(m))=0.$
% Since
%$$A_2=\bigcup_{m=1}^\infty G(m),$$
%we have
% $$\mathcal{L}(A_2)=0\ \text{and} \ \dim_\textsf{H}A_2=1.$$
%Thus (2) follows by noting that $$\liminf_{n\to\infty}\frac{\ell_n(\beta)}{n}=0$$ for any $\beta>1$.

\subsection*{Acknowledgements}
This work was in part finished when the some of the authors visited
the Morning Center of Mathematics (Beijing). The authors are grateful
for the host's warm hospitality. This work was partially supported by
NSFC Nos. 11126071, 10901066, 11225101 and 11171123.

\end{document}